%% file: higgs-duality-paper-adv-accepted.tex
\documentclass{elsarticle}

\usepackage{etoolbox,calc}

\usepackage{appendix}

\usepackage{amsmath, amsthm, amssymb}
\usepackage{enumerate}
\usepackage{verbatim}
\usepackage{url}
\usepackage{graphicx}
\usepackage{geometry}
\usepackage{framed}
\usepackage{tikz-cd}
	\usetikzlibrary{decorations.pathmorphing}

\usepackage{mathrsfs}

\usepackage{mydefs-higgs-dual}

\numberwithin{equation}{section}	

\begin{document}

\begin{frontmatter}

\title{Stacky dualities for the moduli of Higgs bundles.}

\author[mymainaddress,mysecondaryaddress]{Richard Derryberry}
\ead{rderryberry@perimeterinstitute.ca}

\address[mymainaddress]{Perimeter Institute for Theoretical Physics, 31 Caroline Street North, Waterloo, Ontario, Canada, N2L 2Y5}
\address[mysecondaryaddress]{Department of Mathematics, University of Toronto, Bahen Centre, 40 St.\ George St.\ Room 6290, Toronto, Ontario, Canada, M5S 2E5}

\begin{abstract}
The central result of this paper is an identification of the shifted Cartier dual of the moduli stack $\mathcal{M}_{\mathfrak{g}}(C)$ of $\widetilde{G}$-Higgs bundles on $C$ of arbitrary degree (modulo shifts by $Z(\widetilde{G})$) with a quotient of the Langlands dual stack $\mathcal{M}_{^L\mathfrak{g}}(C)$. Via hyperk\"ahler rotation, this may equivalently be viewed as the identification of an SYZ fibration relating Hitchin systems for arbitrary Langlands dual semisimple groups, coupled to nontrivial finite $B$-fields. As a corollary certain self-dual stacks $\frac{\mathcal{M}_{\mathfrak{g}}(C)}{\Gamma}$ are observed to exist, which I conjecture to be the Coulomb branches for the 3d reduction of the 4d $\mathcal{N}=2$ theories of class $\mathcal{S}$.
\end{abstract}


\end{frontmatter}

\setlength{\parindent}{0cm}
\setlength{\parskip}{0.3cm}


\pagestyle{myheadings}
\markright{ \hfill Stacky dualities for the moduli of Higgs bundles.\hfill}
\vspace{-0.8cm}

\tableofcontents

\input{section-introduction}

\input{section-maths-background}

\input{section-stack-duality}

\input{section-examples}

\appendix
\appendixpage
\addappheadtotoc
\renewcommand*{\thesection}{\Alph{section}}

\input{section-appendix-fourier-transform}

\input{section-appendix-Weyl-fixed}

\input{section-appendix-reductive-structure}


\section*{References}

\bibliography{higgs-dual-bib-adv}
\bibliographystyle{elsarticle-num}

\end{document}

%% file: section-introduction.tex

\section{Introduction}\label{s:intro}

The observation that careful analysis of dualities in physics regularly leads to the prediction of novel dualities in mathematics has been responsible for significant advances in modern geometry and representation theory.\footnote{For just the tip of the iceberg: Seiberg-Witten theory and Donaldson invariants \cite{W88,Don90,W94}; 3d $\mc{N}=4$ gauge theories and symplectic duality \cite{BLPW16-II,BDDH16,BFN16-II}; 4d $\mc{N}=4$ Yang-Mills theories and the geometric Langlands program \cite{KapWit07,GW08,DonPan12,EY15,BZN16}.} Over the past decade one particularly rich source of physical dualities have been the ``theories of class $\mc{S}$'' of Gaiotto, Moore and Neitzke \cite{GMN-WKB}. These 4d $\mc{N}=2$ superconformal quantum field theories are obtained from the superconformal 6d $\mc{N}=(2,0)$ theories $\mf{X}_{\mf{g}}$ \cite{Str96,W96} via compactification on a Riemann surface $C$, and are labelled by a choice of simply-laced Lie algebra $\mf{g}$: the resulting 4d theory will be denoted $\mc{S}_{\mf{g}}[C]$. It has previously been observed that these theories exhibit interesting dualities arising from the mapping class group of $C$ \cite{G12}, and that they form the four-dimensional part of the 4d-2d ``AGT correspondence'' \cite{AGT10}. This paper is motivated by a less well-studied self-duality arising from the geometry of the Coulomb branch.

The Coulomb branch of the 4d theory $\mc{S}_{\mf{g}}[C]$ is known to be the Hitchin base $H^0(C;(\mf{g}\git G)\times^{\mb{C}^\times}K_C^\times)$ \cite[\S3.1.2]{GMN-WKB}. General principles imply that the Coulomb branch of the 3d theory obtained via circle compactification is fibred over the 4d Coulomb branch, and by reversing the order in which one compactifies on $C$ and $S^1$ one can argue that the 3d Coulomb branch is given by the moduli space of $G$-Higgs bundles on $C$, $\Higgs_G(C)$ \cite[\S3.1.6]{GMN-WKB}.

It turns out, however, that this description of the 3d Coulomb branch is subtly incorrect. One hint in this direction is the fact that while $\Higgs_G(C)$ requires as data a choice of gauge group $G$, the theory $\mc{S}_{\mf{g}}[C]$ only requires the data of a Lie algebra. More significantly, careful analysis of $\mc{S}_{\mf{g}}[C]$ as a relative quantum field theory \cite{FT14} leads to the conclusion that the Coulomb branch of the 3d theory must be a self mirror-dual space \cite{Tachi14,D-thesis} which requires as extra data a maximally compatible collection of discrete charges for line operators \cite{GMN-BPS}. As $\Higgs_G(C)$ is self mirror-dual only for self Langlands dual groups $G$ \cite{DonPan12}, it cannot be the desired 3d Coulomb branch.

This work originated out of a desire to understand these self-dual moduli spaces, and reasonable candidate spaces $\frac{\mc{M}_{\mf{g}}(C)}{\Gamma}$ are supplied in Corollary \ref{cor:self-dual-stacks}. The spaces $\frac{\mc{M}_{\mf{g}}(C)}{\Gamma}$ are mathematically self-dual and consistent with physical expectations \cite{GMN-BPS,Tachi14}, although there is not yet a direct derivation of these spaces from physical principles.

More broadly, the main results of this paper may be understood in the context of mirror symmetry and Langlands duality for Hitchin systems. S-duality for 4d $\mc{N}=4$ supersymmetric Yang-Mills theory predicts that the Hitchin fibration for the group $G$ will be SYZ mirror dual to the Hitchin fibration for the Langlands dual group $^LG$ \cite{BJSV95,HMS95,KapWit07}; this has been proved for arbitrary reductive $G$ by Donagi and Pantev \cite{DonPan12}, and for type A Hitchin systems coupled to a nontrivial $B$-field by Hausel and Thaddeus \cite{HauTha03}. Theorems \ref{thm:dual-of-M}-\ref{thm:dual-of-intquot} and Corollary \ref{cor:derived-cats} of this paper may be interpreted as an extension of these dualities to incorporate Hitchin systems for arbitrary semisimple groups coupled to nontrivial $B$-fields (i.e.\ equipped with a finite group or $\mc{O}^\times$ gerbe).


\subsection{Summary of mathematical results}\label{ss:math-summary}

The mathematical content of the paper is as follows.

Section \ref{s:maths-background} is devoted to a review of the mathematical background required for the main body of the paper. This background is comprised of two main topics:

Section \ref{ss:CGS} is a review of the notion of shifted Cartier duality for commutative group stacks, a categorification of the notion of ordinary Cartier or Pontrjagin duality where the characters of a group (homomorphisms to $\mb{G}_m$) are replaced by multiplicative line bundles on the group (homomorphisms to $B\mb{G}_m$). Example \ref{eg:BunT-dual} should be paid particular attention in this section, as it reappears in some of the arguments made later in the paper.

Section \ref{ss:higgs-cameral} is a review of certain algebraic aspects of the geometry of the moduli of Higgs bundles, arising from its description as a mapping stack. From this point of view familiar features such as the Hitchin map, Kostant section, interpretation of the Hitchin base in terms of cameral covers, and abelianisation away from the discriminant locus can be understood as induced by the geometry of the adjoint quotient map $\mf{g}\to\mf{h}/W$.

Section \ref{s:stack-dual} contains the primary original contributions of this paper, the main result of which is the construction of a duality between commutative group stacks generalising the moduli stacks of Higgs vector bundles with fixed (nontrivial) determinant \cite{HauTha03}. This duality implies an equivalence of bounded derived categories, and when $\mf{g}$ is of type A we may also derive a certain ``topological mirror symmetry'' statement -- in particular an equality of stringy $E$-polynomials -- by applying the results of \cite{GWZ17,GWZ18}.

Section \ref{ss:compare-isog} contains a comparison of the Hitchin fibres for isogeneous simple groups, showing that for a smooth compact Riemann surface the fibres for isogeneous groups are isogenous abelian varieties. In Section \ref{ss:construction} I construct the moduli stack of primary interest in this paper, $\mc{M}_{\mf{g}}(C)$ -- roughly ``the moduli stack of $\widetilde{G}$-Higgs bundles of arbitrary degree, modulo $Z(\widetilde{G})$'' -- and describe its structure locally over the Hitchin base. In Sections \ref{ss:compareJ} and \ref{ss:bullet-dual} I tease out the likely structure of the dual $\mc{M}_{\mf{g}}(C)^D$ by comparing certain group schemes of regular centralisers and dualising the stack $\HiggsSt_{\widetilde{G}}^\bullet(C)$ of ``$\widetilde{G}$-Higgs bundles of arbitrary degree''. Finally, in Sections \ref{ss:M-dual} and \ref{ss:derived-equivalence} I conclude the proofs of the desired duality theorems via an application of the Langlands duality results of Donagi and Pantev \cite{DonPan12}. The hypotheses under which one obtains a \emph{self}-dual commutative group stack are also made explicit in this section.

Section \ref{s:many-egs} is devoted to four examples of the main duality result, chosen to illustrate different behaviours and compare with results that have appeared previously in the literature. These are: (1) A comparison to the analysis in \cite{GMN-BPS}, which (so far as I am aware) is the first place where it was observed that the Coulomb branch of a class $\mc{S}$ theory compactified on $S^1$ should be a quotient of a moduli space of Higgs bundles depending on a choice of some extra discrete data; (2) a comparison to the SYZ mirror symmetry results of Hausel and Thaddeus \cite{HauTha03}; (3) a comparison to the moduli of Higgs bundles for a self Langlands dual group, which by the results of Donagi and Pantev is self mirror dual space \cite{DonPan12}; (4) an examination of the duality theorem for Lie algebras of type B and C.

Following the main body of the paper, there are three appendices containing supplementary results: Appendix \ref{s:FM} contains an analysis of how the canonical Fourier-Mukai transform acts on the $\mb{Z}\times\mb{Z}$-grading of the derived category of a $\mb{G}_m$-gerbe over a torsor for a commutative group stack; Appendix \ref{s:Weyl} contains a result on when the fixed points of the Weyl group action on a maximal torus may be lifted through an isogeny to another fixed point; and Appendix \ref{s:structure-results} contains results on the structure of the group $\widetilde{G}_\tau$ used in the construction of the moduli space $\mc{M}_{\mf{g}}(C)$ and its Langlands dual.


\subsection{The physical conjecture}\label{ss:physical-conjecture}

Before beginning the mathematical bulk of the paper in earnest, let us consider in slightly more detail the physical situation to which I conjecture it is applicable. The rest of the paper is independent of this section, hence no harm will come to those who wish to skip ahead to the rigorous mathematics.\footnote{On the other hand, those who desire an extended discussion on the topic of this section may be interested in Chapter 2 of \cite{D-thesis}.}


\subsubsection{Some physical background}\label{sss:physical-background}

In order to set a common language, I will begin by setting some terminology. Let $S$ be a structure that can be placed on a manifold\footnote{More accurately, $S$ should collect together different compatible structures for different dimensional manifolds; see \cite{TachiPseudo} for more on this point, as well as for a more detailed discussion of Quasi-Definition \ref{qdef:QFT}.} (e.g.\ smooth structure, Riemannian metric, spin structure, supermanifold structure, $G$-bundle with connection, etc.).

\begin{qdef}[Quantum Field Theory]\label{qdef:QFT}
An \emph{(extended) $d$-dimensional $S$-structured quantum field theory (QFT)}, $\mc{Z}$, is a procedure for functorially assigning
\begin{itemize}
	\item a $\mb{C}$-number $\mc{Z}(M^d)$ to every closed $d$-manifold with structure $S$ (the \emph{correlation function} or \emph{path integral}),
	\item a $\mb{C}$-vector space $\mc{Z}(N^{d-1})$ to every closed $(d-1)$-manifold with structure $S$ (the \emph{space of states}),
	\item a $\mb{C}$-linear category $\mc{Z}(P^{d-2})$ to every closed $(d-2)$-manifold with structure $S$,
	\item higher (appropriately $\mb{C}$-linear) categorical data to higher codimension manifolds with structure $S$,
\end{itemize}
subject to \emph{unitarity} and \emph{locality} constraints.

Furthermore, for every $k<d$ there is a collection of \emph{$k$-dimensional submanifold operators $\{\mc{O}^{(k)}\}$} that may be used to decorate a given manifold, e.g.\ we may evaluate the correlation function of a collection of operators
\begin{align}\label{eq:Z-op-insert}
	\mc{Z}(M^d;\mc{O}^{(k_1)}_{a_1},\ldots,\mc{O}^{(k_l)}_{a_l})\in\mb{C} .
\end{align}
\end{qdef}

\begin{eg}\label{eg:TFT}
It is possible to give a rigorous version of Quasi-Definition \ref{qdef:QFT} in the case where our QFT is a \emph{topological quantum field theory (TQFT)}. In \cite{Lur09} Lurie defines a fully extended topological field theory valued in a symmetric monoidal $(\infty,n)$-category $\mc{C}$ to be a symmetric monoidal functor from a domain bordism $(\infty,n)$-category to $\mc{C}$. Moreover, the \emph{cobordism hypothesis} (due to Baez-Dolan \cite{BaezD95}, Lurie \cite{Lur09}, and others) states that such TQFTs satisfy the strongest possible locality constraint: namely, they are determined by what they evaluate to on a connected 0-manifold (i.e.\ a point).
\end{eg}

\begin{eg}\label{eg:triv-QFT}
The \emph{trivial $d$-dimensional QFT}, $\triv^d$, assigns the number 1 to every $d$-manifold, the 1-dimensional vector space $\mb{C}$ to every $(d-1)$-manifold, the $\mc{C}$-linear category $\Vect_{\mb{C}}$ to every $(d-2)$-manifold, and so on, assigning an $n$-categorical version of a $\mb{C}$-linear symmetric monoidal unit to every $(d-n)$-manifold.
\end{eg}

In order to explain why it is reasonable to conjecture that the moduli spaces $\frac{\mc{M}_{\mf{g}}(C)}{\Gamma}$ are the Coulomb branches of the compactifications of theories of class $\mc{S}$ on $S^1$, it is necessary to generalise the class of theories under consideration to include the \emph{relative field theories} of Freed and Teleman (cf.\ also the ``twisted field theories'' of Stolz and Teichner \cite{ST11}). These are defined in \cite{FT14} as follows: Given a $(d+1)$-dimensional QFT $\alpha$, denote by $\alpha_{\leq d}$ its truncation to manifolds of dimension $\leq d$. Then a \emph{quantum field theory $\mc{Q}$ relative to $\alpha$} is either a homomorphism
\begin{align}\label{eq:coRQFT}
	\mc{Q}: (\triv^{d+1})_{\leq d} \to \alpha_{\leq d} ,
\end{align}
or a homomorphism
\begin{align}\label{eq:contraRQFT}
	\mc{Q}: \alpha_{\leq d}\to (\triv^{d+1})_{\leq d} .
\end{align}
\begin{remark}\label{rk:oplax-theo-claudia}
There are subtleties involved in making this definition precise; for more on this, see the work of Johnson-Freyd and Scheimbauer on defining (op)lax twisted field theories \cite{JFS17}.
\end{remark}

\begin{eg}\label{eg:WZW-anomaly}
The chiral WZW model can be described as a quantum field theory relative to Chern-Simons gauge theory \cite{FT14}. For instance: observe that a relative QFT $\mc{Q}:(\triv^{d+1})_{\leq d} \to \alpha_{\leq d}$ assigns to a $d$-manifold $M^d$ a linear map
\begin{align}\label{eq:RQFT-dmfld}
	\mc{Q}(M^d): (\triv^{d+1})(M^d)=\mb{C} \to \alpha(M^d)
\end{align}
or equivalently (by taking the image of $1\in\mb{C}$) $\mc{Q}(M^d)\in\alpha(M^d)$. Let $\mc{Q}=\mc{Z}_k$ be the chiral WZW model at level $k$ and $\alpha=CS_{G,k}$ be Chern-Simons gauge theory at level $k$. It is shown in \cite{W92} that $\mc{Z}_k(\Sigma;(P,\gd))$ is a gauge invariant holomorphic section of $\mc{L}^{\tens k}$, the ($k^{th}$ tensor power of the) \emph{prequantum line bundle} over the space of all $G$-connections on $\Sigma$; further, this is the space of physical states in 3d Chern-Simons gauge theory at level $k$ \cite{W89}. Hence we have that
\begin{align}\label{eq:WZW-CS}
	\mc{Z}_k(\Sigma)\in H^0(\mc{M}(\Sigma,G);\mc{L}^{\tens k}) = CS_{G,k}(\Sigma) .
\end{align}
\end{eg}

Given a relative QFT, one might try and produce an ``ordinary'' or ``absolute'' QFT from it in the following manner:

\begin{qdef}\label{qdef:absolution}
An \emph{absolution} of a relative QFT $\mc{Q}: (\triv^{d+1})_{\leq d} \to \alpha_{\leq d}$ (resp. $\mc{Q}: \alpha_{\leq d}\to (\triv^{d+1})_{\leq d}$) is another relative QFT $\mc{A}: \alpha_{\leq d}\to (\triv^{d+1})_{\leq d}$ (resp. $\mc{A}: (\triv^{d+1})_{\leq d} \to \alpha_{\leq d}$). If $\mc{A}$ is an absolution of $\mc{Q}$, say that \emph{$\mc{A}$ absolves $\mc{Q}$}.
\end{qdef}

Further discussion of absolution is deferred to the next section, where it will be considered in the context of a specific relative theory, ``Theory $\mf{X}$''.


\subsubsection{Theory $\mf{X}$ and the Coulomb branch}\label{sss:theory-X-coulomb}

One motivation for introducing the formalism of relative QFTs in \cite{FT14} was the desire to understand the structure of a mysterious 6-dimensional theory discovered in \cite{Str96,W96}, particularly those features predicted in \cite[\S4]{W10} which relate to the geometric Langlands program. This theory, known as \emph{Theory $\mf{X}$}, is a 6d (0,2)-superconformal field theory with no classical Lagrangian description.

As explained in \cite[Data 5.1]{FT14}, the data required to specify a Theory $\mf{X}$ is
\begin{enumerate}[(1)]
	\item A real Lie algebra $\mf{g}$ with an invariant inner product $\la -,-\ra$ such that all coroots have square length 2, and
	\item A full lattice $\Gamma$ in a choice of Cartan subalgebra $\mf{h}$, such that $\Gamma$ contains the coroot lattice of $\mf{g}$ and such that $\la-,-\ra$ is integral and even on $\Gamma$.
\end{enumerate}
In what follows I will focus on the case where $\mf{g}$ is semisimple and simply-laced -- in this case the lattice is exactly the coroot lattice $\Pi_R$, and the inner product is a specific normalisation of the Killing form of $\mf{g}$.

\begin{remark}\label{rk:centre-of-group}
Note that if $\widetilde{G}$ is the simply-connected Lie group with Lie algebra $\mf{g}$, the centre of the group may be expressed as $Z(\widetilde{G})=\Pi_W / \Pi_R$ (notation as in \eqref{eq:coroot-coweight}), and the inner product $\la-,-\ra$ induces a symmetric perfect pairing $Z(\widetilde{G})\times Z(\widetilde{G})\to U(1)$.
\end{remark}

Given the above data, \cite[Expectation 5.3]{FT14} predicts the existence
of a 7d TQFT $\alpha_{\mf{g}}$ and a 6d QFT $\mf{X}_{\mf{g}}$ relative to $\alpha_{\mf{g}}$. Explicitly, at the first two category levels: \index{$\alpha_{\mathfrak{g}}$}
\begin{itemize}
	\item To a 6-manifold $X$ $\alpha_{\mf{g}}$ assigns a (finite dimensional) vector space, and the partition function of $\mf{X}_{\mf{g}}$ is a vector $\mf{X}_{\mf{g}}(X)\in\alpha_{\mf{g}}(X)$.
	\item To a 5-manifold $Y$ $\alpha_{\mf{g}}$ assigns a linear category,\footnote{Modelled on {\bf topological} vector spaces.} and the space of states of $\mf{X}_{\mf{g}}$ is an object $\mf{X}_{\mf{g}}(Y)\in\alpha_{\mf{g}}(Y)$.
\end{itemize}
A discussion of the predicted structure of $\alpha_{\mf{g}}$ can be found in \cite[\S5]{FT14} -- I will restrict my discussion here to a description of the partition vector $\mf{X}_{\mf{g}}(X)$ (following \cite{W10,Tachi14}).

Let $X$ be a compact oriented 6-manifold, and consider the middle cohomology group $H^3(X;Z(\widetilde{G}))$. The composition of cup product, the perfect pairing on $Z(\widetilde{G})$, and evaluation against the fundamental class yields a nondegenerate skew pairing
\begin{align}\label{eq:6-mfld-skew}
	\omega:H^3(X;Z(\widetilde{G}))\times H^3(X;Z(\widetilde{G})) \to U(1) .
\end{align}
Such a pairing defines a $U(1)$ central extension known as the \emph{Heisenberg group}, \index{Heisenberg group}
\begin{align}\label{eq:Heisenberg}
	1\to U(1) \to \mc{H}(X,\omega)\to H^3(X;Z(\widetilde{G}))\to 0
\end{align}
characterised by the property that any lifts\footnote{Note that $\Phi$ cannot be a homomorphism (the extension is non-split).} $\Phi(a),\Phi(b)\in\mc{H}(X,\omega)$ of elements $a,b\in H^3(X;Z(\widetilde{G}))$ will satisfy the Heisenberg commutation relation
\begin{align}\label{eq:Hberg-rels}
	\Phi(b)\Phi(a)=\omega(a,b)\Phi(a)\Phi(b) .
\end{align}
The Stone-von Neumann Theorem \cite[Ch.2]{Pol03} states that up to non-canonical isomorphism there is a unique irreducible representation of $\mc{H}(X,\omega)$ on which the central $U(1)$ acts via scalar multiplication. Then $\alpha_{\mf{g}}(X)$ is supposed to be the underlying vector space of this representation.

Here we encounter a problem which, to the best of my knowledge, remains unresolved: namely, to define $\alpha_{\mf{g}}(X)$ it is not sufficient to provide an isomorphism class of vector spaces -- one must specify a \emph{representative} for this isomorphism class. This requires a choice of Lagrangian subgroup\footnote{More accurately, it requires a choice of splitting $L\to\mc{H}(X,\omega)$. \label{fn:splitting}} $L\subset H^3(X;Z(\widetilde{G}))$ (the representation is constructed by considering a class of $L$-invariant functions). Denote the corresponding representation by $\alpha_{\mf{g}}(X;L)$.

Now, given two choices of Lagrangian subgroup\footnote{Satisfying a compatibility condition which depends on the splittings of Footnote \ref{fn:splitting}.} $L_1$ and $L_2$ there is a standard ``Fourier transform'' isomorphism $\alpha_{\mf{g}}(X;L_1)\to\alpha_{\mf{g}}(X;L_2)$, and so one might still hope that $\alpha_{\mf{g}}(X;L)$ is canonically defined. However, given \emph{three} Lagrangian subgroups $L_1$, $L_2$ and $L_3$, the composition
\begin{align}\label{eq:hope-is-lost}
	\alpha_{\mf{g}}(X;L_1)\to\alpha_{\mf{g}}(X;L_2)\to\alpha_{\mf{g}}(X;L_3)\to\alpha_{\mf{g}}(X;L_1)
\end{align}
is not necessarily the identity, but is instead multiplication by some scalar $c(L_1,L_2,L_3)$ \cite[Ch.4]{Pol03}. Hence, absent a choice of Lagrangian subgroup, the canonically defined object is really
\begin{align}\label{eq:canonically-projective}
	\mb{P}\alpha_{\mf{g}}(X) := \mb{P}(\alpha_{\mf{g}}(X;L))
		\quad\text{for any Lagrangian subgroup $L$.}
\end{align}
Following \cite{Tachi14}, I will set this problem aside for the moment in favour of choosing a decomposition $H^3(X;Z(\widetilde{G}))\cong A\oplus B$ where $A,B$ are maximal isotropic (and so in duality with each other via the pairing $\omega$), and choosing splittings $\Phi_A:A\to\mc{H}(X,\omega)$ and $\Phi_B:B\to\mc{H}(X,\omega)$. The action of the elements $\Phi_A(a)$ on $\alpha_{\mf{g}}(X;A)$ may be simultaneously diagonalised by a basis $\{Z_b(X)\}_{b\in B}$ on which the action of $\mc{H}(X,\omega)$ is determined by
\begin{align}\label{eq:basis-action}
	\Phi_A(a)Z_b(X) = \omega(a,b)Z_b(X)
	\quad\text{and}\quad
	\Phi_B(b)Z_{b'}(X) = Z_{b+b'}(X) .
\end{align}
Then the partition vection of Theory $\mf{X}$ (with respect to all the choices we have been forced to make) is given by
\begin{align}\label{eq:X-partition-vector}
	\mf{X}_{\mf{g}}(X) = (Z_b(X))_{b\in B} \in \alpha_{\mf{g}}(X;A) .
\end{align}

Now, suppose that you wanted to absolve Theory $\mf{X}$. Following \cite{Tachi14}, you could choose \emph{another} Lagrangian subgroup $L\subset H^3(X;Z(\widetilde{G}))$ (and splitting $\Phi_L$), not necessarily related to $A$ or $B$. The space of $L$-invariants $\alpha_{\mf{g}}(X;A)^L$ is 1-dimensional, and so the projection of $\mf{X}_{\mf{g}}(X)$ to this subspace gives us an honest partition function $\mf{X}_{\mf{g}}(X;L)$.

\begin{eg}\label{eg:different-theory-x-absolved-partition-fncs}
For $L=A$ the partition function is given by $\mf{X}_{\mf{g}}(X;A)=Z_0$, while for $L=B$ it is given by $\mf{X}_{\mf{g}}(X;B)=\sum_{b\in B} Z_b$.
\end{eg}

This suggests that if one could specify a choice of such a subgroup $L(X)$ in a consistent/functorial manner for all $X$, this might be enough to determine an absolution of Theory $\mf{X}$.

Rather than attempting to do this in full generality, I will restrict to the class of theories relevant to this paper: the \emph{theories of class $\mc{S}$} of Gaiotto, Moore and Neitzke \cite{GMN-WKB}. This class of theories is obtained by compactifying (a particular twist of \cite{FKL,TachiPseudo}) Theory $\mf{X}_{\mf{g}}$ on a Riemann surface $C$. The resulting theory, denoted $\mc{S}_{\mf{g}}[C]$, is still a relative QFT \cite{GMN-BPS,Tachi14}.

As per the above, in order to define an honest partition function for a theory of class $\mc{S}$ on a 4-manifold $M$ we should consider Lagrangian subgroups $L$ of $H^3(M\times C;Z(\widetilde{G}))$. $\omega$ is nondegenerate when restricted to the summands $H^2(M;Z(\widetilde{G}))\tens H^1(C;Z(\widetilde{G}))$ and $H^3(M;Z(\widetilde{G}))\oplus (H^1(M;Z(\widetilde{G}))\tens H^2(C;Z(\widetilde{G})))$ of the K\"unneth decomposition, so we can consider these pieces separately from each other. To specify a Lagrangian subgroup for all 4-manifolds $M$ simultaneously, we choose a Lagrangian subgroup $\Gamma\subset H^1(C;Z(\widetilde{G}))$ and specify $L:=(H^1(M;Z(\widetilde{G}))\tens H^2(C;Z(\widetilde{G})))\oplus H^2(M;Z(\widetilde{G}))\tens\Gamma$.

These choices give rise to an honest QFT $\mc{S}_{\mf{g}}[C,\Gamma]$, whose Hilbert space associated to a 3-manifold of the form $\Sigma\times S^1$, $\Sigma$ a compact oriented surface, is graded by an abelian group containing the summand $H^2(C;Z(\widetilde{G}))$ \cite[(5.14)]{Tachi14}.\footnote{Note that to obtain this grading we have chosen the basis dual to the one chosen in \cite{Tachi14}. For the purposes of this equation I have also ignored the simple-connectivity assumption and instead assumed that the 4d theory is not coupled to a nontrivial background $Z(\widetilde{G})$-torsor via the $Z(\widetilde{G})$ global symmetry.}

The appearance of this summand in the grading may be interpreted as a grading by the second ``Stiefel-Whitney'' classes of $G_{\ad}$-bundles on $C$ \cite[\S6]{Tachi14}, suggesting that the Coulomb branch of the 3d theory should have connected components labelled by $H^2(C;Z(\widetilde{G}))\cong Z(\widetilde{G})$. On the other hand, these Coulomb branches are conjecturally self SYZ mirror dual. The self-duality observation led Tachikawa to conjecture that the Coulomb branch of the 3d theory should be $\Higgs_{\widetilde{G}}(C)/\Gamma$ \cite[\S5.4]{Tachi14}, however this space is connected (and so in particular does not have connected components labeled by $Z(\widetilde{G})$).

In order to incorporate the extra connected components while preserving self-duality the Coulomb branch must be equipped with extra ``stacky'' structure -- namely, it should be a $Z(\widetilde{G})$-gerbe over the course moduli space. In Section \ref{s:stack-dual} I will construct a stack $\mc{M}_{\mf{g}}(C)/\Gamma$ which has precisely this structure, and so it is reasonable to make the following conjecture.

\begin{conj-nonum}
The 3d Coulomb branch of $\mc{S}_{\mf{g}}[C,\Gamma]$ is $\mc{M}_{\mf{g}}(C)/\Gamma$.
\end{conj-nonum}


\subsection{Notation and conventions}\label{ss:conventions}


\subsubsection{Lie theoretic conventions}\label{sss:lie-conventions}

In the following, $G$ is most generally a complex reductive algebraic group, however at times I will note further assumptions of simplicity, simple connectivity, etc. Lie algebras will be denoted by lower case fraktur font, so for instance the Lie algebra of $G$ will be denoted by $\mf{g}$. Given a semisimple group $G$, I will denote by $\widetilde{G}$ the corresponding simply-connected form and by $G_{\ad}$ the corresponding adjoint form. 

A choice of Borel subgroup of $G$ will be denoted $B$, with Lie algebra $\mf{b}$, and a choice of maximal torus will be denoted by $H$ with Lie algebra $\mf{h}$. The notation $T$ is reserved for an algebraic torus that is \emph{not} the maximal torus of a semisimple group $G$, and the (abelian) Lie algebra of such a torus is denoted $\mf{t}$. The rank of a reductive algebraic group $G$ will be denoted by $\rank(G)$, or just by $r$.

When considering the Weyl group associated to a maximal torus $H\subset G$ I will use the notation $W_G(H) = N_G(H)/H$; when I do not need to emphasise the maximal torus $H$ I will just write $W$.

The set of roots of the group $G$ will be denoted by $R$, and a choice of positive roots will be denoted $R^+$. Given a choice of positive roots, the corresponding simple roots will be denoted $S$.

If $M$ is a set or space with a $G$-action I will denote by $M^G$ the fixed points of the $G$-action.

Finally, there are many notations in the literature for the lattices that appear in the study of reductive algebraic groups. As it can sometimes be difficult to keep straight what each piece of notation means (particularly across different references) I have opted to use a notation that makes manifest the input data and the variance for each lattice without being cumbersome. As above, let $T$ denote an algebraic torus, and let $G$ denote a reductive algebraic group with chosen maximal torus $H$:
\begin{itemize}
	\item Denote the character lattice of $T$ by $X^\bullet(T):=\Hom(T,\mb{C}^\times)$, and the cocharacter lattice by $X_\bullet(T):=\Hom(\mb{C}^\times,T)$. When convenient, these can be identified as subgroups $X^\bullet(T) \subset \mf{t}^\ast$ and $X_\bullet(T) \subset \mf{t}$.
	\item Denote by $X^\bullet(G,H):=X^\bullet(H)$ the character lattice corresponding to a choice of maximal torus $H\subset G$; similarly denote the corresponding cocharacter lattice by $X_\bullet(G,H)$. When convenient these can be identified as subgroups $X^\bullet(G,H)\subset \mf{h}^\ast$ and $X_\bullet(G,H)\subset \mf{h}$.
\end{itemize}
When $G$ is semisimple and $H$ is a choice of maximal torus I will denote the root and weight lattices by
\begin{align}\label{eq:root-weight}
	\Lambda_R = X^\bullet(G_{\ad},H_{\ad})
		\quad\text{and}\quad
		\Lambda_W = X^\bullet(\widetilde{G},\widetilde{H})
\end{align}
and the coroot and coweight lattices by
\begin{align}\label{eq:coroot-coweight}
	\Pi_R = X_\bullet(G_{\ad},H_{\ad}) = \Lambda_R^{\wed} = \Hom_{\mb{Z}}(\Lambda_R,\mb{Z})
		\quad\text{and}\quad
		\Pi_W = X_\bullet(\widetilde{G},\widetilde{H}) = \Lambda_W^{\wed}.
\end{align}


\subsubsection{Geometric conventions}\label{sss:geometry-conventions}

A general complex scheme or manifold will be denoted by $X$, with structure sheaf $\mc{O}_X$, and a general test scheme will be denoted $S$. Throughout most of the paper, the constant sheaf on $X$ valued in $A$ is denoted $A_X$. (The exception occurs between Lemma \ref{lemma:J-surj} and Example \ref{eg:surj-fails}, where this notation will refer to a different sheaf in accordance with the notation used in \cite{DonGai02}.) The notation $C$ will be reserved for the situation where the space in question is a Riemann surface or an algebraic curve (usually, but not always, of genus $g>1$).

Given a space $X$ and spaces equipped with maps to $X$, $Y_1\to X$ and $Y_2\to X$, I will denote by $\Hom_X(Y_1,Y_2)$ the collection of maps $Y_1\to Y_2$ in the slice category of spaces with a map to $X$.

Given a group $G$, I will use the algebro-geometric terminology \emph{$G$-torsor} to refer to a principal $G$-bundle. I.e.\ a $G$-torsor over a space $X$ is a space $P\to X$ equipped with a (right) $G$-action, such that (1) the map $(\id_P,\text{act}):P\times G \to P\times_X P$ is an isomorphism and (2) $P$ admits local sections. Here the terms ``space'' and ``local'' are deliberately vague, as this definition is applicable to many different categories and Grothendieck topologies.

As a general rule, stacky moduli spaces are denoted via calligraphic and italic fonts, while coarse moduli spaces are denoted via bold font. Stacky quotients are denoted by square brackets $[\,/\,]$: if $X$ is equipped with a right action of $G$, then $[X/G]$ denotes the stack with presentation given by the groupoid \cite[\S2.4.3]{LMB00}
\begin{equation}\label{diag:stack-quot-present}
\begin{tikzcd}
	X\times G \ar[d, shift right,"s"']\ar[d,shift left,"t"]	\\
	X
\end{tikzcd} \quad s(x,g)=x,\quad t(x,g)=x\cdot g.
\end{equation}
Given two stacks $\mc{Y}$ and $\mc{Z}$, I will denote by $\MapSt(\mc{Y},\mc{Z})$ the sheaf of groupoids whose $S$ points are given by $\Map_S(\mc{Y}\times S, \mc{Z}\times S)$ for any affine scheme $S$. Similarly, if $\mc{A}$, $\mc{B}$ are commutative group stacks, I will denote by $\HomSt(\mc{A},\mc{B})$ the commutative group stack whose $S$-points are given by $\Hom_S(\mc{A}\times S,\mc{B}\times S)$ for any affine scheme $S$ \cite[XVIII]{SGA4c}.

Finally and importantly: from Important Remark! \ref{rk:good-locus} onwards, I will implicitly restrict away from the discriminant locus of the Hitchin base (see Definition \ref{def:Hitchin-base} and \eqref{eq:discr}). The duality results of Section \ref{s:stack-dual} will hold over this dense open set of $\Hitch_{\mf{g}}(C)$ -- the question of whether or not this duality may be extended over the discriminant is still open. Partial results in this direction have been obtained by Arinkin and Fedorov \cite{A13,AF16}.


\subsubsection{Duality conventions}\label{sss:duality-conventions}
This paper involves significant interplay between various well-known dualities. To distinguish between them I use the following notation:
\begin{itemize}
	\item $^L(-)$ denotes an object obtained via Langlands duality, e.g.\ the Langlands dual group $^L G$.
	\item $(-)^\vee$ denotes the Pontrjagin dual group $\Hom(-,U(1))$ or $\Hom(-,\mb{G}_m)$, depending on context.
	\item $(-)^{\wed}$ denotes the dual lattice to an abelian group, $(-)^{\wed}:=\Hom(-,\mb{Z})$.
	\item $(-)^D$ denotes the shifted Cartier dual $\Hom(-,\mc{O}^\times[1])$ or $\Hom(-,B\mb{G}_m)$, depending on context. E.g.\ if $A$ is an abelian variety then $A^D$ is the usual dual abelian variety.
\end{itemize}
\begin{remark}\label{rk:alg-to-syz}
To relate the algebraic concept of shifted Cartier duality for the Hitchin fibration to the usual notion of an SYZ fibration (i.e.\ in terms of special Lagrangian fibrations and flat $U(1)$-connections \cite{SYZ96}), simply perform a hyperk\"ahler rotation on the moduli space of semistable Higgs bundles \cite{Hit97}.
\end{remark}


\subsubsection*{Acknowledgements.} I wish to thank my Ph.D.\ advisors David Ben-Zvi and Andrew Neitzke, who originally suggested I study these self-dual spaces, and who have provided me with invaluable advice and guidance. I also thank all those with whom I have discussed this project, in particular Michael Groechenig for his comments on topological mirror symmetry, and Ron Donagi, whose feedback helped clarify and streamline some of the arguments. Finally, I thank the anonymous referees for their helpful comments and suggestions.

This paper was begun while I was supported as a graduate student at the University of Texas at Austin, and was completed with support from the Perimeter Institute for Theoretical Physics and the University of Toronto. Research at Perimeter Institute is supported in part by the Government of Canada through the Department of Innovation, Science and Economic Development Canada and by the Province of Ontario through the Ministry of Economic Development, Job Creation and Trade. I also acknowledge support from U.S.\ National Science Foundation grants DMS 1107452, 1107263, 1107367 ``RNMS: Geometric Structures and Representation Varieties” (the GEAR Network).''

\label{a}

%% file: section-maths-background.tex

\section{Review of Cartier duality and Higgs bundles}\label{s:maths-background}

To begin, let us briefly review the mathematical concepts central to the paper: (1) shifted Cartier duality for commutative group stacks, and (2) the moduli of Higgs bundles and the Hitchin fibration.


\subsection{Commutative group stacks and shifted Cartier duality}\label{ss:CGS}


\subsubsection{Definition and examples}\label{sss:CGS-def-eg}

Categorical background, e.g.\ material on symmetric monoidal categories, may be found in \cite{MacL71,KSch06}. Background on stacks and descent theory may be found in \cite{LMB00,Vi05}. Material on shifted Cartier duality may be found in \cite{CZ17,Bro14,Camp17} as well as in Arinkin's appendix to \cite{DonPan08}. As always, $k$ denotes an algebraically closed field.

\begin{defn}\label{def:Pic-gpd}
A \emph{Picard groupoid} is a symmetric monoidal category in which every object is invertible (with respect to the monoidal structure) and every morphism is invertible (in the usual sense).
\end{defn}

\begin{remark}\label{rk:equiv-class}
Given a Picard groupoid $(\mc{C},\tens)$ the set of equivalence classes of objects $\pi_0\mc{C}$ is a commutative group in a canonical way.
\end{remark}

The canonical example of a Picard groupoid, which in particular explains the nomenclature, is as follows:

\begin{eg}\label{eg:Picard}
Let $X$ be a complex manifold, and consider the category whose objects are holomorphic line bundles on $X$ and whose morphisms are given by isomorphisms of holomorphic line bundles. Tensor product of line bundles endows this category with the structure of a Picard groupoid, and the commutative group obtained by taking $\pi_0$ is exactly the \emph{Picard group} of holomorphic line bundles on $X$.
\end{eg}

\begin{defn}\label{def:CGS}
Let $X$ be a space endowed with a Grothendieck topology. A \emph{commutative group stack on $X$} is a sheaf of Picard groupoids on $X$.
\end{defn}

\begin{remark}\label{rk:CGS-sites}
I have left the meaning of ``space'' in Definition \ref{def:CGS} deliberately ambiguous. In this paper I will primarily work with complex varieties with the analytic or \'etale topology (cf.\ \cite{DonPan08,DonPan12}), although the material in this section applies in much greater generality (e.g.\ algebraic stacks equipped with the fppf topology \cite[XVIII 1.4]{SGA4c}, \cite{Bro14,Camp17}).
\end{remark}

\begin{eg}\label{eg:Hom-stack}
Given two commutative group stacks $\mc{A}$ and $\mc{B}$ over $X$ there is a commutative group stack $\HomSt(\mc{A},\mc{B})$ whose $U$-points are given by the category $\Hom_U(\mc{A}\times_X U,\mc{B}\times_X U)$ \cite[Def. 2.4 \& Ex. 2.8]{Bro14}.
\end{eg}

\begin{eg}\label{eg:ab-var}
Given a $k$-scheme $X$, an \emph{abelian scheme} over $X$ is a smooth group scheme over $X$ whose fibres are abelian varieties (group schemes which are complete varieties over $k$).
\end{eg}

\begin{eg}\label{eg:ab-sheaf}
Any sheaf of abelian groups $K$ over $X$ may be regarded as a commutative group stack with discrete objects (and trivial automorphisms).
\end{eg}

\begin{eg}\label{eg:class-stack}
Given a sheaf of abelian groups $K$ over $X$, the classifying stack $BK$ whose $U$-points are $BK(U)=(\text{groupoid of $K|_U$-torsors on $U$})$ is a commutative group stack.
\end{eg}

\begin{remark}\label{rk:cx-of-sheaves}
There is a convenient reformulation of the theory of commutative group stacks in terms of complexes of sheaves, due to Deligne \cite[XVIII, 1.4]{SGA4c}. Let $\Ch^{[-1,0]}(X)$ denote the 2-category given by:
\begin{itemize}
	\item Objects are complexes of abelian sheaves on $X$ concentrated in degrees -1 and 0, $A^\bullet = [A^{-1}\to A^0]$, such that $A^{-1}$ is injective.\footnote{The injectivity assumption implies that the quotient prestack $[A^0/A^{-1}]$ is already a stack.}
	\item Morphisms are chain maps of complexes.
	\item 2-morphisms are homotopies of chain maps.
\end{itemize}
Given a complex of abelian sheaves of the form $A^{-1}\to A^0$ the quotient stack $[A^0/A^{-1}]$ is a commutative group stack on $X$. This construction gives an equivalence between $\Ch^{[-1,0]}(X)$ and the 2-category of commutative group stacks on $X$ \cite[XVIII, 1.4.17]{SGA4c}. This may be interpreted as a (length 1) form of the Dold-Kan correspondence between simplicial objects and chain complexes.
\end{remark}


\subsubsection{Shifted Cartier duality}\label{sss:Cartier-duality}

Given an abelian variety $A$ the \emph{dual abelian variety} $A^D$ is the moduli space of multiplicative line bundles on $A$. Recalling that $B\mb{G}_m$ is the classifying stack for $\mb{G}_m$-torsors, i.e.\ algebraic line bundles, the ``multiplicative'' condition may be translated into the statement that $A^D := \HomSt(A,B\mb{G}_m)$. This example may be generalised as follows:

\begin{defn}\label{def:Cartier-dual}
Let $\mc{A}$ be a commutative group stack over $X$. The \emph{shifted Cartier dual of $\mc{A}$} is the commutative group stack $\mc{A}^D:=\HomSt(\mc{A},B\mb{G}_m)$.
\end{defn}

\begin{remark}\label{rk:Cartier-analytic}
When working over $\mb{C}$ in the analytic topology, the definition/notation $\mc{A}^D=\HomSt(\mc{A},B\mc{O}^\times)$ is sometimes used, cf.\ \cite{DonPan08,DonPan12}.
\end{remark}

\begin{defn}\label{def:reflexive}
A commutative group stack $\mc{A}$ is \emph{reflexive} if the canonical morphism $\mc{A}\to(\mc{A}^D)^D$ is an isomorphism.
\end{defn}

\begin{remark}\label{rk:term-origin}
The terminology ``reflexive'' is adopted after \cite{Camp17}; the same property is termed ``dualisability'' in \cite{Bro14}.
\end{remark}

\begin{eg}[Dualising sheaves and classifying stacks {\cite[Cor.\ 3.5{-}6]{Bro14}}]\label{eg:dual-BK}
Let $K$ be a sheaf of abelian groups on $X$. Then:
\begin{enumerate}
	\item There is a canonical isomorphism $(BK)^D \simeq K^\vee=\HomSt(K,\mb{G}_m)$.
	\item There is a canonical homomorphism $B(K^\vee)\to K^D$, which is an isomorphism if $\ExtSh^1(K,\mb{G}_m)=0$.
\end{enumerate}
In particular, if $K$ is locally finitely generated then $K$ is a reflexive commutative group stack.
\end{eg}

\begin{defn}\label{def:exact}
Given a commutative group stack $\mc{A}$ over $X$ there are two associated sheaves of abelian groups \cite[Def. 2.9]{Bro14}:
(1) the coarse moduli sheaf $\pi_0(\mc{A})$, and (2) the automorphism group of a neutral section $\pi_1(\mc{A})$. A sequence of commutative group stacks $\mc{A}\to\mc{B}\to\mc{C}$ is \emph{exact} if both sequences of sheaves of abelian groups
\begin{align}
	\pi_0(\mc{A})\to\pi_0(\mc{B})\to\pi_0(\mc{C})	\label{eq:pi0-exact} \\
	\pi_1(\mc{A})\to\pi_1(\mc{B})\to\pi_1(\mc{C})	\label{eq:pi1-exact}
\end{align}
are exact.
\end{defn}

The following proposition is immediate:
\begin{prop}\label{p:autoequivalence}
Shifted Cartier duality is an exact, contravariant, involutive autoequivalence on the 2-category of reflexive commutative group stacks.
\end{prop}

One might fear that the hypotheses of Proposition \ref{p:autoequivalence} are too restrictive to apply to any interesting examples. The following proposition, together with reflexivity of abelian varieties and Example \ref{eg:dual-BK} proves that we need not worry:

\begin{prop}\label{p:local-reflexive}
Suppose that a commutative group stack $\mc{A}$ over $X$ is locally isomorphic to a product of reflexive commutative group stacks. Then $\mc{A}$ is a reflexive commutative group stack.
\end{prop}
\begin{proof}
The canonical map $\mc{A}\to(\mc{A}^D)^D$ is an isomorphism if and only if it is an isomorphism locally on $X$ -- but this is exactly our hypothesis (cf.\ \cite[Prop.\ A.6]{DonPan08} and \cite[Appendix A]{CZ17}).
\end{proof}

\begin{eg}\label{eg:BunT-dual}
Let $T$ be an algebraic torus, and let $X$ be a smooth, projective, connected curve over $k$. The moduli stack of $T$-bundles on $X$ is the commutative group stack $\BunSt_T(X)=\MapSt(X,BT)$. Denote by $\Bun_T(X)$ the corresponding coarse moduli space, and by $\BunSt^0_T(X)$ and $\Bun^0_T(X)$ the corresponding neutral components. These are all commutative group stacks, with the following shifted Cartier duals:
\begin{align}
	\BunSt_T(X)^D 		&= \BunSt_{^LT}(X)	\label{eq:BunSt-dual}	\\
	\BunSt_T^0(X)^D 	&= \Bun_{^LT}(X)   	\label{eq:BunSt0-dual}	\\
	\Bun_T^0(X)^D 		&= \Bun_{^LT}^0(X)	\label{eq:Bun0-dual}
\end{align}
\end{eg}


\subsection{Higgs bundles and cameral covers}\label{ss:higgs-cameral}

Fix a Riemann surface (or complex smooth projective algebraic curve) $C$, which I will often assume to have genus $>1$, and denote by $K_C\to C$ the canonical bundle of $C$. Let $G$ be a complex reductive algebraic group. The following (standard) notion of a Higgs bundle is attributable to Hitchin \cite{Hit87a,Hit87b}:
\begin{defn}\label{defn:valHiggs}
A \emph{$K_C$-valued $G$-Higgs bundle on $C$} is a pair $(E,\varphi)$, where
\begin{itemize}
	\item $E\to C$ is a holomorphic $G$-bundle, and 
	\item $\varphi \in H^0(C;\ad(E)\tens K_C)$, i.e.\ $\varphi$ is a global section of the bundle $\ad(E)\tens K_C$.
\end{itemize}
Here, $\ad(E)$ is the vector bundle associated to $E$ via the adjoint representation of $G$ on $\mf{g}=\Lie(G)$.
\end{defn}
\begin{remark}\label{rk:L-val-Higgs}
By replacing $K_C$ with any other line bundle $L\to C$, we obtain the more general notion of an \emph{$L$-valued $G$-Higgs bundle on $C$}.
\end{remark}
Consider the Lie algebra $\mf{g}$ of $G$ as a $G\times\mb{G}_m$-module, via the adjoint action of $G$ and the scaling action of the multiplicative group. A map from a scheme $X$ to the stack quotient $[\mf{g}/G\times\mb{G}_m]$ is given by the data of
\begin{enumerate}[(1)]
	\item a principal $G$-bundle $E$ on $X$, and
	\item a line bundle (i.e.\ principal $\mb{G}_m$-bundle) $L$ on $X$, together with
	\item a section of the vector bundle $\ad(E)\tens L \to X$.
\end{enumerate}
Hence the stack $\MapSt(X,[\mf{g}/G\times\mb{G}_m])$ is exactly the moduli stack of $G$-Higgs bundles on $X$ with values in \emph{some} line bundle. Composition of such a map with the natural projection map $[\mf{g}/G\times\mb{G}_m] \to B\mb{G}_m$ classifies the line bundle of values for the corresponding Higgs bundle.

\begin{defn}\label{def:Higgs-stack}
Denote by $\HiggsSt_G(X,L)$ the \emph{moduli stack of $G$-Higgs bundles on $X$ with values in $L$}, i.e.\ the substack of $\MapSt(X,[\mf{g}/G\times\mb{G}_m])$ whose projection to $\MapSt(X,B\mb{G}_m)=\BunSt_{\mb{G}_m}(X)$ classifies the line bundle $L$.
\end{defn}

\begin{remark}\label{rk:coarse-Higgs-moduli}
There is also a moduli space of semistable $L$-valued $G$-Higgs bundles on $C$ \cite{Hit87a,Nit91,Simp94II}, which I will denote by $\Higgs_G(C,L)$. There is a natural open substack of semistable $L$-valued $G$-Higgs bundles on $C$, $\HiggsSt^{ss}_G(C,L) \subset \HiggsSt_G(C,L)$ which maps to this coarse moduli space, and whose image under the Hitchin map (Definition \ref{def:Hitchin-map}) is contained in the very regular locus (Definition \ref{def:very-reg}).
\end{remark}


\subsubsection{The Hitchin fibration and sections}\label{sss:Hitchin-fib-sec}

The moduli of Higgs bundles admits a canonical description as a fibration, admitting a canonical class of sections. Let us briefly recall these structures, and how they are induced from the representation theory of $G$.

Fix the data of a maximal torus and a Borel subgroup $H \hookrightarrow B \subset G$. This determines a set of simple roots $S$ in the root system $R$ of the Lie algebra $\mf{g}$. Denote the root space of $\mf{g}$ corresponding to $\alpha\in R$ by $\mf{g}_\alpha$, and choose a nonzero vector $x_\alpha\in \mf{g}_\alpha$ for each simple $\alpha\in S$. For each simple root $\alpha$ there is then a unique element $x_{-\alpha}\in\mf{g}_{-\alpha}$ such that $[x_\alpha,x_{-\alpha}]=\alpha^\vee$, the coroot corresponding to $\alpha$. Then the elements $x_+ = \sum_{\alpha\in S}x_\alpha$ and $x_- = \sum_{\alpha\in S}x_{-\alpha}$ are regular\footnote{Recall that the \emph{locus of regular elements} $\mf{g}^{reg}\subset\mf{g}$ is the locus of elements whose centralisers have minimal possible dimension $\dim(Z_G(x))=\rank(G)$.} nilpotent elements of $\mf{g}$ \cite{Ko59}.

Consider the adjoint action of $G$ on $\mf{g}$, and the induced action of the Weyl group $W := W_G(H)=N_G(H)/H$ on $\mf{h}=\Lie(H)$. These induce actions of $G$ and $W$ on the algebras $\bC[\mf{g}]$ and $\bC[\mf{h}]$ respectively, and we define $\mf{c}:=\Spec(\bC[\mf{h}]^W)=\mf{h}/W$. We then have the following theorem of Kostant \cite{Ko63} and Chevalley:

\begin{thm}\label{thm:Kostant-section}
\begin{enumerate}
	\item[]
	\item The restriction map $\bC[\mf{g}]\to\bC[\mf{h}]$ induces an isomorphism on the subalgebras of invariants $\bC[\mf{g}]^G \stackrel{\sim}{\to}\bC[\mf{h}]^W$. Moreover, $\bC[\mf{h}]^W$ is a polynomial algebra generated by homogeneous elements $P_1,\ldots,P_r$ of degrees $m_1+1,\ldots,m_r+1$.
	\item The \emph{Chevalley} or \emph{characteristic polynomial map} $\chi:\mf{g}\to\mf{c}$, induced by the above isomorphism, is $\mb{G}_m$-equivariant with respect to the weight one action of $\mb{G}_m$ on $\mf{g}$, and the action on $\mf{c}$ defined by
		\[	\lambda\cdot(P_1,\ldots,P_r)=(\lambda^{m_1+1}P_1,\ldots,\lambda^{m_r+1}P_r).	\]
	\item The restriction of $\chi$ to the regular locus $\mf{g}^{reg}\subset\mf{g}$ is smooth, and each fibre is a single $G$-orbit.
	\item Let $\mf{g}^{x_+}\subset\mf{g}$ denote the Lie algebra centraliser of $x_+$ (i.e.\ the kernel of $\ad(x_+)$ acting on $\mf{g}$). Then the affine subspace $x_- + \mf{g}^{x_+}$ is contained in the regular locus $\mf{g}^{reg}$, and the Chevalley map restricts to an isomorphism $x_- + \mf{g}^{x_+} \cong \mf{c}$.
\end{enumerate}
\end{thm}

\begin{remark}\label{rk:Kostant-section}
The inverse to the Chevalley map on $x_- + \mf{g}^{x_+}$ is called the \emph{Kostant section}, and will be denoted by $\kappa$:
\begin{equation}\label{eq:Kostant-section}
\begin{tikzcd}[column sep = small]
	\mf{g}^{reg}\ar[r, hook]	& \mf{g}\ar[d,"\chi"]	\\
						& \mf{c}\ar[ul,"\kappa"]
\end{tikzcd}
\end{equation}
\end{remark}

For $L\to X$ a line bundle consider the associated $\mf{c}$-bundle on $X$
\begin{align}\label{eq:twist-c}
	\mf{c}_L := \mf{c}\times^{\mb{G}_m} (L-0) = (L\tens\mf{h})/W.
\end{align}

\begin{defn}\label{def:Hitchin-base-functor}
$\HitchSt_{\mf{g}}(X,L)$ is the functor whose $S$-points are given by
\begin{align}\label{eq:Hitchin-base-functor}
	\Hom(S,\HitchSt_{\mf{g}}(X,L))=\Hom_X(S\times X,\mf{c}_L).
\end{align}
\end{defn}

Since by definition $\mf{c}=\mf{h}/W=\mf{g}\git G$, the Chevalley map factors through the stack $[\mf{g}/G]$. Since it is $\mb{G}_m$-equivariant, it further descends to give maps $[\mf{g}/G\times\mb{G}_m] \to [\mf{c}/\mb{G}_m]$ and
\begin{align}\label{eq:Chev-descent-2}
	\MapSt(X,[\mf{g}/G\times\mb{G}_m]) \to \MapSt(X,[\mf{c}/\mb{G}_m]).
\end{align}

\begin{defn}\label{def:Hitchin-map}
The restriction of \eqref{eq:Chev-descent-2} to the subfunctor classifying the line bundle $L$ is the \emph{Hitchin map}
\begin{align}\label{eq:Hitchin-map}
	h_L:	\HiggsSt_G(X,L)\to \HitchSt_{\mf{g}}(X,L).
\end{align}
\end{defn}

\begin{defn}\label{def:Hitchin-section}
The Kostant section \eqref{eq:Kostant-section} induces a section
\begin{align}\label{diag:Hitchin-section}
\begin{tikzcd}
	\HiggsSt_G(X,L)\ar[d,"h_L"]	\\
	\HitchSt_{\mf{g}}(X,L)\ar[u,bend left,"s_L"]
\end{tikzcd}
\end{align}
known as the \emph{Hitchin section}. This section may depend upon a choice of square-root $L^{1/2}$ for $L$; see \cite{Hit92} for details.
\end{defn}


\subsubsection{Cameral covers}\label{sss:cameral-covers}

\begin{defn}\label{def:Hitchin-base}
The \emph{Hitchin base} $\Hitch_{\mf{g}}(X,L)$ is the $\bC$-points of the Hitchin base functor, i.e.\ $\Hitch_{\mf{g}}(X,L)=\Hom_X(X,\tot(\mf{c}_L)) = H^0(X;(L\tens\mf{h})/W)$.
\end{defn}

The Hitchin base $\Hitch_{\mf{g}}(X,L)$ is an affine space, and it represents the functor $\HitchSt_{\mf{g}}(X,L)$. Moreover, it parametrises the \emph{$L$-valued cameral covers of $X$}, which we define after \cite{DonPan12} as follows:

\begin{defn}\label{def:cameral-cover}
A \emph{cameral cover} of $X$ is a scheme $\tilde X$ together with a map $p:\tilde X\to X$ and a $W$-action along the fibres of $p$ satisfying:
\begin{enumerate}
	\item $p$ is finite and flat over $X$.
	\item As an $\mc{O}_X$-module with $W$-action $p_*\mc{O}_{\tilde X}$ is locally isomorphic to $\mc{O}_X\tens\bC[W]$.
	\item Locally with respect to the \'etale (or analytic) topology on $X$, $\tilde X$ is a pullback of the $W$-cover $\mf{h}\to\mf{h}/W$.
\end{enumerate}
An \emph{$L$-valued cameral cover} of $X$ is a cameral cover $p:\tilde{X}\to X$ together with a $W$-equivariant map $\tilde{\sigma}:\tilde{X}\to \tot(L\tens\mf{h})$.
\end{defn}


\subsection{The group scheme of regular centralisers}\label{ss:J}

Following the ideas of \cite{DonGai02,Ngo10}, I will now review a uniform approach to understanding the fibres of the Hitchin map \eqref{eq:Hitchin-map} via the ``group scheme of regular centralisers''.


\subsubsection{Regular centralisers over the adjoint quotient}\label{sss:J-c}

Consider first the \emph{group scheme of centralisers} $I\to\mf{g}$ defined by
\begin{align}\label{eq:I-scheme}
	I = \{(x,g)\in\mf{g}\times G \,|\, \Ad_g(x)=x\}\subset \mf{g}\times G.
\end{align}
This map is very poorly behaved: observe for instance that it interpolates between the fibre of a regular semisimple element, which is an algebraic torus of dimension $r$, and the fibre over 0, which is a copy of $G$. When restricted to the regular locus, however, $I^{reg}$ becomes a smooth commutative group scheme of relative dimension $r$, whose generic fibre (over a semisimple element) is an algebraic torus.

Recalling that the Kostant section $\kappa$ is valued in $\mf{g}^{reg}\subset\mf{g}$, we may make the following definition.

\begin{defn}\label{def:J-scheme}
The \emph{group scheme of regular centralisers $J$} is the pullback
\begin{align}\label{eq:J-scheme}
	J = \kappa^* I^{reg} \to \mf{c}.
\end{align}
\end{defn}

Since $I^{reg}$ is a smooth commutative group scheme, so is $J$. Consider the pullback by the Chevalley map $\chi^*J \to \mf{g}$: by construction this is equipped with an isomorphism over the regular locus $(\chi^* J)|_{\mf{g}^{reg}} \stackrel{\sim}{\to} I|_{\mf{g}^{reg}}$, and this extends uniquely to a homomorphism of group schemes $\chi^* J \to I$ since $J$ is smooth, $I$ is affine, and $\chi^*J \setminus \chi^* J |_{\mf{g}^{reg}}$ is closed of high codimension \cite{Ngo10}. $J$ descends to a group scheme over $[\mf{c}/\mb{G}_m]$, and in fact $[\mf{g}^{reg}/G]\to \mf{c}$ is a $J$-gerbe, trivialised by the Kostant section \cite{Ngo06}.


\subsubsection{Regular centralisers over the Hitchin base}\label{sss:J-Hitch}

Let us now consider a Picard stack on the Hitchin base $\Hitch_{\mf{g}}(X,L)$, which we define following \cite{Ngo06}. Recall that a point $\sigma:S\to\Hitch_{\mf{g}}(X,L)$ is equivalent to a map
\begin{align}\label{eq:Hitchin-S-point}
	h_\sigma:X\times S \to [\mf{c}/\mb{G}_m]
\end{align}
which lies over the map $X\to B\mb{G}_m$ that classifies the line bundle $L\to X$. By pulling back the smooth commutative group scheme $J\to [\mf{c}/\mb{G}_m]$ along $h_\sigma$, we obtain a smooth family of commutative group schemes $J_\sigma = h_\sigma^*J \to X\times S$. 

Now consider the category of $J_\sigma$-torsors on $X\times S$, $\Tors_{J_\sigma}(X\times S)$. The assignment $\sigma\mapsto\Tors_{J_\sigma}(X\times S)$ defines a Picard stack on $\Hitch_{\mf{g}}(X,L)$, denoted $\TorsSt_J$.

\begin{prop}\label{prop:J-tors-Higgs-action}
There is an action of $\Tors_{J_\sigma}(X\times S)$ on the fibre $\HiggsSt_G(X,L)_\sigma:=h_L(S)\inv(\sigma)$ of the ($S$-points of the) Hitchin map.
\end{prop}

Through this action we may interpret the moduli of Higgs bundles $\HiggsSt_G(X,L)$ as a partial compactification of $\TorsSt_J$ as follows. Write $\HiggsSt_G(X,L)^{reg}$ for the subfunctor classifying those maps $h_{E,\phi}:X\times S \to [\mf{g}\tens L / G]$ which factor through the open substack $[\mf{g}^{reg}\tens L / G]$. Assume that $L$ admits a square root, so that the Hitchin fibration admits a Kostant section. Since by construction the Kostant section takes values in the regular locus, we have the following \cite[Prop 4.3.3]{Ngo10}:

\begin{prop}\label{prop:regular-Higgs-locus}
$\HiggsSt_G(X,L)^{reg}$ is open in $\HiggsSt_G(X,L)$ with non-empty fibres over $\Hitch_{\mf{g}}(X,L)$. Moreover, $\TorsSt_J$ acts on this locus simply-transitively.
\end{prop}


\subsubsection{The Hitchin fibration away from the discriminant locus}\label{sss:discr-loc}

Consider the branch locus of the generically \'etale Galois $W$-cover $\mf{h}\to\mf{c}$, denoted by $\mf{D}_{\mf{g}}$. This may be identified with the divisor given by vanishing of the discriminant
\begin{align}\label{eq:discr}
	\prod_{\alpha\in R} d\alpha ,
\end{align}
where the product is over the roots of $G$.\footnote{Since $\alpha: H\to\mb{G}_m$, $d\alpha: \mf{h}\to\mb{G}_a$.} We adopt the follow definition after \cite{Ngo06}:

\begin{defn}\label{def:very-reg}
Call $\sigma\in\Hitch_{\mf{g}}(X,L)$ \emph{very regular} if the image of the associated map $h_\sigma: X\to \mf{c}_L$ is transverse to the divisor $\mf{D}_L = \mf{D}_{\mf{g}}\times^{\mb{G}_m} L$.
\end{defn}

\begin{remark}\label{rk:very-reg-geom}
Geometrically, Definition \ref{def:very-reg} means that the associated cameral cover $p_\sigma:\tilde{X}_\sigma \to X$ has simple Galois ramification, i.e.\ all of the ramification points of $p_a$ have ramification index one \cite{DonPan12}. Moreover in this situation $\tilde{X}$ is smooth \cite{Ngo10}.
\end{remark}

If $L$ is very ample then the very regular locus is open and dense in $\Hitch_{\mf{g}}(X,L)$ \cite{Ngo06}. Denote the complement of this locus by $\Delta_{\mf{g}}$ (or just $\Delta$ if $\mf{g}$ is clear from context), so that the very regular locus is $\Hitch_{\mf{g}}(X,L)\setminus \Delta$.

\begin{prop}\cite[Prop 4.3]{Ngo06}\label{prop:Higgs-J-gerbe}
For $\sigma\in\Hitch_{\mf{g}}(X,L)\setminus\Delta$, the groupoid $\Tors_{J_\sigma}(X\times S)$ acts simply-transitively on $\HiggsSt_G(X,L)_\sigma$; i.e.\ $\HiggsSt_G(X,L)_\sigma$ is a $J_\sigma$-gerbe. Moreover, if the Hitchin section exists, it trivialises this gerbe.
\end{prop}

Finally, let $(X,L)=(C,K_C)$, and recall the coarse moduli space of semistable $K_C$-valued Higgs bundles $\Higgs_{G}(C,K_C)$. The above groupoid level analysis, together with the fact that Higgs bundles with very regular characteristics are stable, yields the following corollary upon passage to equivalence classes:

\begin{cor}\label{cor:H1J}
The Hitchin fibre $\Higgs_G(C,K_C)_\sigma$ lying over a very regular characteristic $\sigma\in\Hitch_{\mf{g}}(C,K_C)$ is a torsor for $H^1(C; J_\sigma)$, the group of equivalence classes of $J_\sigma$-torsors on $C$.
\end{cor}

\label{b}

%% file: section-stack-duality.tex

\section{Duality for quotients of the moduli of Higgs bundles}\label{s:stack-dual}

In this section I present new duality results relating moduli stacks of Higgs bundles for Langlands dual groups. These build on previous results by Donagi and Pantev, who proved Langlands duality of Hitchin systems for arbitrary reductive groups \cite{DonPan12}, and Hausel and Thaddeus who proved a duality result for Hitchin systems of type A equipped with an extra stacky structure (the ``gerbe of liftings'') \cite{HauTha03}. The goal of this section is to generalise the results of Hausel and Thaddeus to arbitrary semisimple groups; the existence of certain self-dual moduli stacks will then follow as a corollary.



\subsection{Comparison of Hitchin fibres for isogenous simple groups}\label{ss:compare-isog}

In what follows I will make heavy use of comparisons between Hitchin Pryms (Definition \ref{def:HPrym}) for different reductive groups belonging to the same isogeny class. Let $G$ be a reductive algebraic group, $C$ be a compact Riemann surface, and denote by $J^G_\sigma \to C$ the pullback of the group scheme of regular centralisers for $G$ by the map $\sigma: C\to [\mf{c}/\mb{G}_m]$ classifying a point in the Hitchin base; i.e.\ $J^G_\sigma = \sigma\inv J^G$, where $J^G$ is the group scheme of regular centralisers for $G$ (Definition \ref{def:J-scheme}).

Restrict to the situation where $G$ a simple group. The following claim may be checked locally:

\begin{lemma}\label{l:J-SES}
Let $G\to G/Z$ denote an isogeny of simple groups, so that $Z$ is a discrete subgroup of the centre $Z(G)$. There is a short exact sequence of commutative group schemes over $[\mf{c}/\mb{G}_m]$,
\begin{align}\label{eq:J-universal-SES}
	0 \to Z_{\mf{c}} \to J^G \to J^{G/Z} \to 0.
\end{align}
\end{lemma}

Since pullback of sheaves is exact there is an analogous exact sequence over any other $[\mf{c}/\mb{G}_m]$-scheme. In particular, corresponding to a point $\sigma$ in the Hitchin base we have a short exact sequence of sheaves over $C$
	\begin{align}\label{eq:J-SES}
		0 \to Z_C \to J^G_\sigma \to J^{G/Z}_\sigma \to 0 .
	\end{align}


Suppose now that $\widetilde{G}$ is a connected and simply-connected simple group. Taking the long exact sequence of \eqref{eq:J-SES} yields
\begin{equation}\label{eq:J-LES}
\begin{tikzcd}
		0\ar[r]
			&Z \arrow[r]
				& \Gamma(C;J^{\widetilde{G}}_\sigma) \arrow[r] \arrow[d, phantom, ""{coordinate, name=Z}]
					& \Gamma(C;J^{\widetilde{G}/Z}_\sigma) \arrow[dll,rounded corners,to path={ -- ([xshift=2ex]\tikztostart.east)|- (Z) [near end]\tikztonodes-| ([xshift=-2ex]\tikztotarget.west)-- (\tikztotarget)}] \\
			&H^1(C;Z) \arrow[r]
				& H^1(C;J^{\widetilde{G}}_\sigma) \arrow[r] \arrow[d, phantom, ""{coordinate, name=Y}]
					& H^1(C;J^{\widetilde{G}/Z}_\sigma)\arrow[dll,rounded corners,to path={ -- ([xshift=2ex]\tikztostart.east)|- (Y) [near end]\tikztonodes-| ([xshift=-2ex]\tikztotarget.west)-- (\tikztotarget)}] \\
			&H^2(C;Z)	& \phantom{0}
	\end{tikzcd}
\end{equation}

\begin{defn}\label{def:HPrym}
For simply connected $\widetilde{G}$ and very regular characteristic $\sigma$, the \emph{Hitchin Prym for $\widetilde{G}$ associated to $\sigma$} is
\begin{align}\label{eq:scPrym}
	H^1(C;J^{\widetilde{G}}_\sigma)
	\quad(\cong\Higgs_{\widetilde{G}}(C)_\sigma) .
\end{align}
For a general reductive group $G$ define the \emph{Hitchin Prym} to be the identity component $H^1(C;J^G_\sigma)_0 \subset H^1(C;J^G_\sigma)$.
\end{defn}

\begin{remark}\label{rk:HPrym-semisimple}
Note that for a non simply-connected semisimple group $\widetilde{G}/Z$ the Hitchin Prym is given by
\begin{align}\label{eq:Prym}
	\ker[H^1(C;J^{{\widetilde{G}}/Z}_\sigma)\to H^2(C;Z)]
	\quad(\cong\Higgs^0_{\widetilde{G}/Z}(C)_\sigma) .
\end{align}
\end{remark}

\begin{remark}\label{rk:implicit-triv}
In order to identify the cohomology group $H^1(C;J^G_\sigma)$ with the Hitchin fibre $\Higgs_G(C)_\sigma$ I have implicitly trivialised the gerbe of Higgs bundles \cite{DonGai02} using a Hitchin section \eqref{diag:Hitchin-section}.
\end{remark}

The Hitchin Pryms are known to be abelian varieties \cite{DonPan12}, and a rephrasing of Corollary \ref{cor:H1J} yields that the fibres of the Hitchin fibration for $\widetilde{G}/Z$ which lie over very regular characteristics (Definition \ref{def:very-reg}) are torsors for the $H^1(C;J_\sigma^{\widetilde{G}/Z})_0$.

Rewrite the exact sequence associated to \eqref{eq:J-SES} as
\begin{equation}\label{eq:Prym-LES}
\begin{tikzcd}
		0\ar[r]
			&Z \arrow[r]
				& \Gamma(C;J^{\widetilde{G}}_\sigma) \arrow[r] \arrow[d, phantom, ""{coordinate, name=Z}]
					& \Gamma(C;J^{\widetilde{G}/Z}_\sigma) \arrow[dll,rounded corners,to path={ -- ([xshift=2ex]\tikztostart.east)|- (Z) [near end]\tikztonodes-| ([xshift=-2ex]\tikztotarget.west)-- (\tikztotarget)}] \\
			&H^1(C;Z) \arrow[r]
				& \Higgs_{\widetilde{G}}(C)_\sigma \arrow[r] 
					& \Higgs^0_{\widetilde{G}/Z}(C)_\sigma \ar[r]
						& 0.
	\end{tikzcd}
\end{equation}

Recall that we denote by $R$ the set of roots of the group $\widetilde{G}$. 
\begin{lemma}\label{lemma:J-surj}
Suppose that $C$ is a smooth, proper, irreducible curve over $\bC$, that the line bundle classified by $\sigma$ has nontrivial $|R|^{th}$ power, and that $\sigma$ is a very regular characteristic. Then the map on global sections $\Gamma(C;J^{\widetilde{G}}_\sigma)\to \Gamma(C; J^{\widetilde{G}/Z}_\sigma)$ is surjective.
\end{lemma}

\begin{remark}\label{rk:alt-J}
In what follows I will make use of an alternative and more explicit description of the sheaf of regular centralisers, which is due to \cite{DonGai02}. Denote by $\pi_\sigma:\tilde{C}_\sigma \to C$ the cameral cover of $C$ classified by $\sigma:C\to[\mf{c}/\mb{G}_m]$, and consider the sheaf on $\tilde{C}_\sigma$ of holomorphic maps to a choice of maximal torus $H\subset G$, $H(\mc{O}_{\tilde{C}_\sigma})$. Push this sheaf down to $C$ and take $W:= W_G(H)$-invariants, calling the result $\overline{H}_{\tilde{C}_\sigma}$,
\begin{align}\label{eq:pushforward-invariants}
\overline{H}_{\tilde{C}_\sigma}(U) = \left((\pi_\sigma)_\ast H(\mc{O}_{\tilde{C}_\sigma})^W \right)(U) = \Hom_W(\tilde{U}_\sigma,H) ,
\end{align}
i.e.\ $W$-equivariant maps from the induced cameral cover $\tilde{U}_\sigma$ to the maximal torus $H$. Denote by $\mf{D}_\sigma^\alpha$ the fixed point scheme of the root reflection $s_\alpha \in W$ acting on $\tilde{C}_\sigma$, and define a subsheaf\footnote{Note that this is a break from the convention that this should mean the constant sheaf on $\tilde{C}_\sigma$ valued in $H$ -- this break in convention will persist until the end of Example \ref{eg:surj-fails}.} $H_{\tilde{C}_\sigma}\subset\overline{H}_{\tilde{C}_\sigma}$ by
\begin{align}\label{eq:T-sheaf}
	H_{\tilde{C}_\sigma}(U) = \{t\in\overline{H}_{\tilde{C}_\sigma}(U)\,|\, (\alpha\circ t)|_{\mf{D}_\sigma^\alpha} = +1 \text{ for each } \alpha\in R\}.
\end{align}
Then according to \cite[Theorem 11.6]{DonGai02} there is an isomorphism between $J^G_\sigma$ and $H_{\tilde{C}_\sigma}$. I will use the description given by the latter in the proof of Lemma \ref{lemma:J-surj}.
\end{remark}

\begin{proof}[Proof of Lemma \ref{lemma:J-surj}]
Observe that since $C$ is proper so is $\tilde{C}_\sigma$, and since $\sigma$ is assumed to be very regular $\tilde{C}_\sigma$ is non-singular. Thus, since $H=\widetilde{H}/Z$ is affine, any map from $\tilde{C}_\sigma$ to $H$ will be locally constant; i.e.\
\begin{align}\label{eq:T-bar-const}
	\overline{H}_{\tilde{C}_\sigma}(C) = \Hom_W(\tilde{C}_\sigma,H) = \Hom_W(\pi_0(\tilde{C}_\sigma),H).
\end{align}
First consider the case where $\tilde{C}_\sigma$ is connected -- for instance, this is true if $g(C)>1$ and $L=K_C$ \cite{SCO98}  -- so that $\overline{H}_{\tilde{C}_\sigma}(C) = (\widetilde{H}/Z)^W$. Then
\begin{align}\label{eq:H-global-sections}
	H_{\tilde{C}_\sigma}(C)
		&= \{hZ\in (\widetilde{H}/Z)^W \,|\, \alpha(hZ) = +1, \,\forall\alpha\in R \text{ s.t. } s_\alpha\text{ fixes some point in }\tilde{C}_\sigma\} 	\\
		&= \{hZ\in (\widetilde{H}/Z)^W \,|\, s_\alpha(hZ) = hz, \,\forall hz\in hZ, \,\forall\alpha\in R \text{ s.t. } s_\alpha\text{ fixes some point in }\tilde{C}_\sigma \} , \nonumber
\end{align}
where the second equality follows from Proposition \ref{thm:fixed-lifts}. The discriminant \eqref{eq:discr}, which locally detects ramification of cameral covers, pulls back along $\sigma$ to give a section of the $|R|^{th}$ power of the line bundle classified by $\sigma$. By assumption this is non-trivial, hence we can guarantee the existence of a root\footnote{In the simply-laced case, or a short and a long root in the non-simply laced case by the corresponding factorisation of \eqref{eq:discr} into a product over short and long roots.} $\alpha$ such that $s_\alpha$ fixes some point in $\tilde{C}_\sigma$. Via the $W$-action, we can therefore guarantee that \emph{every} root $\alpha\in R$ fixes at least one point of $\tilde{C}_\sigma$. Since the Weyl group acts trivially on the centre, we have
\begin{align}\label{eq:H-ram-global-sections}
	H_{\tilde{C}_\sigma}(C)
		&= \{hZ\in(\widetilde{H}/Z)^W \,|\, s_\alpha(h')=h', \,\forall h'\in hZ \} = \widetilde{H}^W/Z = \widetilde{H}_{\tilde{C}_\sigma}(C)/Z.
\end{align}
If $\tilde{C}_\sigma$ has multiple connected components, choose one and denote it by $\tilde{C}_\ast$. Setting $S:=\Stab_W([\tilde{C}_\ast])$, $[\tilde{C}_\ast]\in\pi_0(\tilde{C}_\sigma)$, we may identify $\overline{H}_{\tilde{C}_\sigma}(C) = H^S$. Then the same argument as above goes through, using that $S$ is generated by those $s_\alpha\in W$ such that $s_\alpha$ fixes some point in $\tilde{C}_\ast$.
\end{proof}

\begin{eg}\label{eg:surj-fails}
How could this have failed? Suppose that $C$ is an irreducible complex projective variety that admits a connected \'etale double cover: all double covers are $\mf{sl}_2\bC$ cameral covers, so we are implicitly assuming that our double cover is cameral and valued in some line bundle which has trivial square. Since there is no ramification the condition \eqref{eq:T-sheaf} is vacuous and so
\begin{align}\label{eq:J-PGL-J-SL}
	J^{SL_2}(C) = H_{SL_2}^W\cong\bZ/2\bZ \quad\text{and}\quad J^{PGL_2}(C)=H_{PGL_2}^W=\bZ/2\bZ,
\end{align}
where the Weyl group invariants in this case are calculated in Example \ref{eg:Winvt-sl2}. Thus, in this example $J^{PGL_2}(C) \not\cong J^{SL_2}(C)/Z(SL_2\bC)$.
\end{eg}

From Lemma \ref{lemma:J-surj} we obtain a comparison theorem relating any Hitchin Prym to the Hitchin Prym for the connected simply-connected group:

\begin{thm}\label{thm:PrymComp}
Let $\widetilde{G}$ be a simple, connected, simply-connected group and $\widetilde{G}\to \widetilde{G}/Z$ an isogeny. Then for $\sigma\in\Hitch_{\mf{g}}(C,K_C)\setminus\Delta_{\mf{g}}$ there is an isomorphism of abelian varieties
\begin{align}\label{eq:Prym-quot}
	\Higgs^0_{\widetilde{G}/Z}(C)_\sigma  = \frac{\Higgs_{\widetilde{G}}(C)_\sigma}{H^1(C,Z)}.
\end{align}
\end{thm}
\begin{proof}
By Lemma \ref{lemma:J-surj}, $H^1(C;Z)\to \Higgs_{\widetilde{G}}(C)_\sigma$ is injective, thus the long exact sequence of \eqref{eq:J-SES} involving the Hitchin Pryms breaks up into two short exact sequences; the isomorphism of the theorem is the content of the bottom sequence.
\end{proof}

\begin{remark}\label{rk:large-genus}
To really get value out of Theorem \ref{thm:PrymComp} one should assume that the genus of $C$ is at least 2, so that the very regular locus is open and dense in the Hitchin base \cite{Falt93}.
\end{remark}

\begin{thm}\label{thm:dual-prym}
Let $\widetilde{G}$ be a simple, connected, simply-connected group, and let $\widetilde{^LG}$ denote the simply-connected cover of its Langlands dual group. Then
\begin{align}\label{eq:dual-Prym-quot}
	(\Higgs_{\widetilde{^LG}}(C)_\sigma)^D = \frac{\Higgs_{\widetilde{G}}(C)_\sigma}{H^1(C;Z(\widetilde{^LG}))^\vee} .
\end{align}
\end{thm}
\begin{proof}
By \cite[Theorem A]{DonPan12} we have that $\Higgs_{\widetilde{G}}(C)_\sigma = (\Higgs^0_{(^LG)_{\ad}}(C)_\sigma)^D$. Dualising the isogeny of abelian varieties from Theorem \ref{thm:PrymComp}
\begin{align}\label{eq:Prym-isogeny}
	0\to H^1(C;Z(\widetilde{^LG}))\to \Higgs_{\widetilde{^LG}}(C)_\sigma \to \Higgs^0_{(^LG)_{\ad}}(C)_\sigma \to 0
\end{align}
we obtain the dual isogeny
\begin{align}\label{eq:dual-Prym-isogeny}
	0\to H^1(C;Z(\widetilde{^LG}))^\vee \to (\Higgs^0_{(^LG)_{\ad}}(C)_\sigma)^D\to (\Higgs_{\widetilde{^LG}}(C)_\sigma)^D\to 0 .
\end{align}
\end{proof}


\subsection{Construction and local structure of $\HiggsSt_{\widetilde{G}}^\bullet(C)$ and $\mc{M}_{\widetilde{G}}^\bullet(C)$}\label{ss:construction}

In their proof of Langlands duality for $SL/PGL$-Hitchin systems \cite{HauTha03}, Hausel and Thaddeus make use not just of the moduli of $SL_n$-Higgs bundles but of the moduli space of ``degree $d$'' $SL_n$-Higgs bundles. This does not literally make sense as written (as an $SL_n$-bundle has trivial determinant and is thus degree zero) -- what is meant by this is ``$GL_n$-Higgs bundles with determinant a \emph{fixed line bundle of degree $d$} and \emph{trace-free Higgs field}''.

To generalise the results of \cite{HauTha03,DonPan12}, and to prove the existence of a self-dual space, I will now construct a generalisation of this space for $\widetilde{G}$-Higgs bundles, where $\widetilde{G}$ may be any connected simply-connected semisimple group (cf.\ \cite{BLS98} for an analogous construction for the moduli stack of bundles).


\subsubsection{Construction of $\HiggsSt_{\widetilde{G}}^\bullet(C)$}\label{sss:construct-Higgs-bullet}

Let $\mu_N$ denote the group of $N^{th}$ roots of unity with generator $\omega := e^{\frac{2\pi i}{N}}$. Observe that a homomorphism $\tau:(\mu_N)^s \to (\mb{C}^\times)^s$ is determined by an $s\times s$-matrix $A=(A_{ji})\in\Mat_{s\times s}(\mb{Z}/N\mb{Z})$ by setting
\begin{equation}\label{eq:tau-matrix}
\begin{tikzcd}[row sep = tiny]
	(\mu_N)^s \ar[r,"\tau"]			& (\mb{C}^\times)^s 	\\
	(\omega^{\vec{a}})\ar[r,mapsto]	& (\omega^{A\vec{a}})
\end{tikzcd}
\end{equation}
where $\vec{a}\in(\mb{Z}/N\mb{Z})^s$ and $(\omega^{\vec{a}})=(\omega^{a_1},\ldots,\omega^{a_s})\in(\mu_N)^s$.

\begin{defn}\label{def:special-emb}
Call an homomorphism $\tau:(\mu_N)^s \to (\mb{C}^\times)^s$ a \emph{general embedding} if it can be represented by a matrix in the image of the map $GL_s(\mb{Z})\to GL_s(\mb{Z}/N\mb{Z})$.

More generally, let $K$ be a finite abelian group equipped with an isomorphism $k:K\simeq\mu_{N_1}\times\cdots\times\mu_{N_s}$ with $s$ minimal, and let $T$ be a complex algebraic torus of rank $s$. I will call a homomorphism $\tau:K\to T$ a \emph{general embedding} if the map $\tau\circ k\inv:(\mu_{lcm(N_1,\ldots,N_s)})^s \to (\mb{C}^\times)^s$ is a general embedding for some isomorphism $T\simeq (\mb{C}^\times)^s$.
\end{defn}

\begin{remark}\label{rk:spec-emb-dont-care-bout-T-triv}
If $\tau$ is a general embedding with respect to \emph{some} isomorphism $T\simeq (\mb{C}^\times)^s$, then since automorphisms of $(\mb{C}^\times)^s$ correspond to elements of $GL_s(\mb{Z})$, it is in fact a general embedding with respect to \emph{all} such isomorphisms.
\end{remark}

\begin{remark}\label{rk:spec-emb-is-emb}
It is not difficult to show that $\tau:(\mu_N)^s \to (\mb{C}^\times)^s$ is an embedding if and only if any matrix $A$ which represents it is in $GL_s(\mb{Z}/N\mb{Z})$. In particular, general embeddings are embeddings.
\end{remark}

\begin{eg}\label{eg:simple-group-embeddings}
Suppose that we are interested in general embeddings of the centre of a simple group. There are two possibilities:
\begin{itemize}
	\item The centre is cyclic, isomorphic to $\mu_N$. In this case, embeddings correspond to elements of $(\mb{Z}/N\mb{Z})^\times$, while there are only two general embeddings, $\omega\mapsto \omega^{\pm 1}$ (distinct if $N\neq 2$).
	\item The centre is $\mu_2\times\mu_2$. In this case all embeddings are general embeddings, given by elements of $GL_2(\mb{F}_2) \simeq S_3$ (the symmetric group on 3 letters).
\end{itemize}
\end{eg}

Now, let $\widetilde{G}$ be a connected simply-connected simple group with centre $Z(\widetilde{G})$, fix a trivialisation $k:Z(\widetilde{G})\to\mu_{N_1}\times\cdots\times\mu_{N_s}$ with $s$ minimal, and let $\tau:Z(\widetilde{G})\to T$ be a general embedding of $Z(\widetilde{G})$ into a complex algebraic torus (whose rank $s$ is necessarily equal to the number of cyclic factors in the trivialisation of $Z(\widetilde{G})$, by the definition of a general embedding).\footnote{This value of $s$ is moreover the minimal possible rank for a torus admitting an embedding of $Z(\widetilde{G})$.}

\begin{defn}\label{def:G-tau}
Define a group $\widetilde{G}_\tau$ by the equation
\begin{align}\label{eq:G-tau}
	\widetilde{G}_\tau := \frac{\widetilde G \times T}{Z(\widetilde{G})} ,
\end{align}
where $Z(\widetilde{G})\subset \widetilde{G}$ is the inclusion homomorphism.
\end{defn}

\begin{prop}\label{prop:tau-indept}
The group $\widetilde{G}_\tau$ is independent of the choice of general embedding, up to non-canonical isomorphism.
\end{prop}

\begin{proof}
Suppose that $\tau_1$, $\tau_2$ are two general embeddings, and consider them as maps from $(\mu_N)^s\to(\mb{C}^\times)^s$ where $N$ is the lowest common multiple of the orders of the cyclic factors of $Z(\widetilde{G})$. Let $A_1$, $A_2$ be representative matrices for the general embeddings. We wish to find an automorphism $\beta:(\mb{C}^\times)^s\to(\mb{C}^\times)^s$ such that $\beta\circ\tau_1=\tau_2$.

As observed above, $\beta$ will be represented by some matrix $B\in\Mat_{s\times s}(\mb{Z})$. For $\beta\circ\tau_1=\tau_2$ to hold, we need that for all $\vec{a}\in(\mb{Z}/N\mb{Z})^s$, $\beta\circ\tau_1(\omega^{\vec{a}})=(\omega^{BA_1\vec{a}})=(\omega^{A_2\vec{a}})=\tau_2(\omega^{\vec{a}})$, which occurs if and only if $BA_1 \equiv A_2$ modulo $N$. But by the definition of a general embedding the matrices representing $\tau_1$ and $\tau_2$ may be lifted to matrices in $SL_s(\mb{Z})$, which I will also denote by $A_1$ and $A_2$, and so it suffices to take $B=A_2A_1\inv$.

To complete the proof of the proposition, it suffices to observe that $[\id_{\widetilde{G}}\times\beta]$ is a well-defined isomorphism $\widetilde{G}_{\tau_1}\simeq \widetilde{G}_{\tau_2}$.
\end{proof}

\begin{remark}\label{rk:constraint}
It is reasonable to ask whether we really needed to consider general embeddings, or whether any matrix $A\in GL_s(\mb{Z}/N\mb{Z})$ would suffice. In fact, we do: Suppose that $\det(A)\neq\pm 1$ modulo $N$, so that $A$ cannot be lifted to $GL_s(\mb{Z})$. It is possible to find an automorphism $\gamma$ of $(\mu_N)^s$, represented by a matrix $C$, such that $\det(AC)=1$. In order for this to induce an isomorphism as in Proposition \ref{prop:tau-indept}, $\alpha$ would need to extend to an automorphism of the group $\widetilde{G}$, necessarily not an inner automorphism. But, for example, $\Out(SL_8\mb{C})=\mb{Z}/2\mb{Z}$ while $\Aut(Z(SL_8\mb{C}))=\Aut(\mb{Z}/8\mb{Z})=\mb{Z}/2\mb{Z}\times\mb{Z}/2\mb{Z}$ -- so there are necessarily automorphisms of the centre which do not extend to automorphisms of the entire group.
\end{remark}

\begin{remark}\label{rk:canon}
Note that the isomorphism $\beta$ in Proposition \ref{prop:tau-indept} is {\bf not} unique. For example, if $s=2$ we have $\left(\begin{array}{c c} 1 & 0 \\ 0 & 1 \end{array}\right)\equiv\left(\begin{array}{c c} 1 & N \\ 0 & 1 \end{array}\right)$, and both are in $GL_2(\mb{Z})$.
\end{remark}

The group $\widetilde{G}_\tau$ comes equipped with two projections
\begin{equation}\label{diag:G-tau-proj}
\begin{tikzcd}
			& \widetilde{G}_\tau \ar[dl,"p"']\ar[dr,"\pd"] 	\\
	G_{\ad} 	&	& T/Z(\widetilde{G})
\end{tikzcd}
\end{equation}

Note that $T/Z(\widetilde{G}) \simeq T$ non-canonically: for the moment I will not choose such an isomorphism.

\begin{eg}\label{eg:SL-tau-is-GL}
Let $\widetilde{G}=SL_n$ and $\tau:Z(SL_n)=\mu_n \subset \mb{G}_m$.\footnote{Although some results will require that we work over $\mb{C}$, many of the constructions -- such as this one -- are independent of the ground ring.} Then $\widetilde{G}_\tau = GL_n$ and the maps $p$ and $\pd$ are
\begin{equation}\label{diag:GL-proj}
\begin{tikzcd}
			& GL_n \ar[dl,"p"']\ar[dr,"\det"] 	\\
	PGL_n 	&	& \mb{G}_m
\end{tikzcd}
\end{equation}
\end{eg}

The Lie algebra of $\widetilde{G}_\tau$ is
 \begin{align}\label{eq:G-tau-Lie}
 	\mf{g}_\tau = \mf{g}\oplus\mf{t}
 \end{align}
 where $\mf{g}=\Lie(\widetilde{G})$ and $\mf{t}=\Lie(T)$. Let $H\subset\widetilde{G}$ be a maximal torus with Lie algebra $\mf{h}$ so that
 \begin{align}\label{eq:H-tau}
 	H_\tau = \frac{H\times T}{Z(\widetilde{G})}
 \end{align}
 is a maximal torus of $\widetilde{G}_\tau$ with Lie algebra $\mf{h}_\tau=\mf{h}\times\mf{t}$. Since $\mf{t}$ is abelian the quotient $\mf{c}_\tau=\mf{h}_\tau/W$ is
\begin{align}\label{eq:adj-quot-tau}
	\mf{c}_\tau = (\mf{h}/W)\times \mf{t} = \mf{c}\times\mf{t}
\end{align}
where $\mf{c}=\mf{h}/W$ is the adjoint quotient for the group $\widetilde{G}$ and $W\equiv W_{\widetilde{G}_\tau}(H_\tau)=W_{\widetilde{G}}(H)$ is the Weyl group. Thus there is a ``Hitchin map'' between stacks (cf.\ \eqref{eq:Chev-descent-2})
\begin{align}\label{eq:Hitchin-tau}
	\chi_\tau = \chi\times\id_{\mf{t}}:
			\left[\mf{g}_\tau/(\widetilde{G}_\tau\times\mb{G}_m)\right]
			=\left[(\mf{g}\times\mf{t})/(\widetilde{G}_\tau\times\mb{G}_m)\right]
			\to\left[\mf{c}/\mb{G}_m\right]\times\left[ \mf{t}/\mb{G}_m \right]
			=\left[\mf{c}_\tau/\mb{G}_m\right] .
\end{align}

The maps $p$ and $\pd$ induce maps
\begin{equation}\label{diag:Higgs-target-proj}
\begin{tikzcd}[column sep = tiny]
			& \left[(\mf{g}\times\mf{t})/(\widetilde{G}_\tau\times\mb{G}_m)\right] \ar[dl,"p_*"']\ar[dr,"\pd_*"] 	\\
	\left[\mf{g}/G_{\ad}\times\mb{G}_m\right] 	&	& \left[\mf{t}/(T/Z(\widetilde{G}))\times\mb{G}_m\right]\simeq B(T/Z(\widetilde{G}))\times\left[\mf{t}/\mb{G}_m\right]
\end{tikzcd}
\end{equation}
and so for a space $X$ there are maps
\begin{equation}\label{diag:Higgs-mapping-proj}
\begin{tikzcd}[column sep = tiny]
			& \MapSt\left(X,\left[(\mf{g}\times\mf{t})/\widetilde{G}_\tau\times\mb{G}_m\right]\right) \ar[dl,"p_*"']\ar[dr,"\pd_*"] 	\\
	\MapSt\left(X,\left[\mf{g}/G_{\ad}\times\mb{G}_m\right]\right) 	&	& \BunSt_{T/Z(\widetilde{G})}(X)\times\MapSt\left(X,\left[\mf{t}/\mb{G}_m\right]\right)
\end{tikzcd}
\end{equation}
Supposing now that the pushforwards to $B\mb{G}_m$ all classify the line bundle $L\to X$, we obtain maps
\begin{equation}\label{diag:Higgs-stack-proj}
\begin{tikzcd}
			& \HiggsSt_{\widetilde{G}_\tau}(X,L) \ar[dl,"p_*"']\ar[dr,"\pd_*"] 	\\
	\HiggsSt_{G_{\ad}}(X,L) 	&	& \BunSt_{T/Z(\widetilde{G})}(X)\times H^0(X;\mf{t}\tens L)
\end{tikzcd}
\end{equation}

\begin{eg}\label{eg:GL-Higgs-stack-proj}
In the running $SL_n/GL_n$ example \eqref{diag:GL-proj}, these maps are
\begin{equation}\label{diag:Higgs-stack-proj-GLn}
\begin{tikzcd}
			& \HiggsSt_{GL_n}(X,L) \ar[dl,"p_*"']\ar[dr,"\det\times\tr"] 	\\
	\HiggsSt_{PGL_n}(X,L) 	&	& \PicSt(X)\times H^0(X;L)
\end{tikzcd}
\end{equation}
\end{eg}

Now, choose an isomorphism $t:T\cong \mb{G}_m^s$ such that $Z(\widetilde{G})$ is sent to a product of groups of roots of unity -- this is possible by composing a trivilisation which exhibits $\tau$ as a general embedding with an automorphism of $\mb{G}_m^s$ that exhibits it as a product of standard embeddings $\mu_{i}\to\mb{G}_m$ -- so that $t$ induces an isomorphism
\begin{align}\label{eq:torus-triv}
	T/Z(\widetilde{G}) \cong \frac{\mb{G}_m}{\mu_{i_1}}\times\cdots\times\frac{\mb{G}_m}{\mu_{i_s}}
\end{align}
and by taking $i_j^{\text{th}}$ powers componentwise we obtain an isomorphism $T/Z(\widetilde{G})\cong\mb{G}_m^s$. This isomorphism of groups allows us to further identify
\begin{align}\label{eq:Bun-Pic-stack}
	\BunSt_{T/Z(\widetilde{G})}(X)\cong \PicSt(X)\times\cdots\times\PicSt(X) .
\end{align}

Now, suppose that $X=C$ is a connected Riemann surface, or a smooth connected complex projective algebraic curve. Choose a point $x\in C$ and for $\vec{p}=(p_1,\ldots,p_s)$ denote
\begin{align}\label{eq:Opx}
	\mc{O}(\vec{p}x)=(\mc{O}(p_1 x),\ldots,\mc{O}(p_s x))\in\BunSt_{T/Z(\widetilde{G})}(C)
\end{align}
where we have implicitly used the isomorphism \eqref{eq:Bun-Pic-stack}. Define a lattice by $\Lambda(x)=\{\mc{O}(\vec{p}x)\,|\,\vec{p}\in\mb{Z}^s\}\subset\BunSt_{T/Z(\widetilde{G})}(C)$. Passing to the group of connected components of $\BunSt_{T/Z(\widetilde{G})}(C)$ exhibits an isomorphism
\begin{equation}\label{diag:Opx-lattice}
	\begin{tikzcd}
		\Lambda(x)\ar[rr,"\cong"]\ar[rd,"\iota"] 	& & X_\bullet(T/Z(\widetilde{G})) \\
			& \BunSt_{T/Z(\widetilde{G})}(C)\ar[ru]
	\end{tikzcd}
\end{equation}
and so yields a splitting $\iota_x:X_\bullet(T/Z(\widetilde{G}))\hookrightarrow \BunSt_{T/Z(\widetilde{G})}(C)$.

\begin{defn}\label{def:Higgs-bullet}
Define $\HiggsSt_{\widetilde{G}}^\bullet(C,L)$ to be the pullback of stacks over the \emph{trace-free locus} of the Hitchin base $\{0\}\subset H^0(C;\mf{t}\tens L)$
\begin{equation}\label{diag:Higgs-bullet-def}
\begin{tikzcd}
	\HiggsSt_{\widetilde{G}}^\bullet(C,L)\ar[rr]\ar[d]\arrow[drr, phantom,"\lrcorner", very near end]
		& & \HiggsSt_{\widetilde{G}_\tau}(C,L)\ar[d,"\pd_*"] \\
	X_\bullet(T/Z(\widetilde{G}))\ar[r,"\iota_x"']
		& \BunSt_{T/Z(\widetilde{G})}(C) \ar[r,hook,"\id\times 0"] 	& \BunSt_{T/Z(\widetilde{G})}(C)\times H^0(C;\mf{t}\tens L)
\end{tikzcd}
\end{equation}
\end{defn}

\begin{remark}\label{rk:x-indept}
Note that given another point $y\in C$, the embeddings $\iota_x$ and $\iota_y$ differ by the automorphism of $\BunSt_{T/Z(\widetilde{G})}(C)$ given by tensoring with the $T/Z(\widetilde{G})$-bundle $\mc{O}(\vec{p}(y-x))$ on the component of $\BunSt_{T/Z(\widetilde{G})}$ labelled by $\vec{p}$. Since $\mc{O}(\vec{p}(y-x))\in\BunSt_{T/Z(\widetilde{G})}^0(C)$ we may lift them to $T$-bundles $\tilde{\mc{O}}_{\vec{p}}\in\BunSt_T^0(C)$, additive in $\vec{p}$. These furnish an automorphism of $\HiggsSt_{\widetilde{G}_\tau}(C,L)$ by the action of $\BunSt_T(C)$ on $\HiggsSt_{\widetilde{G}_\tau}(C,L)$, since $T$ is central in $\tilde{G}_\tau$. By uniqueness of pullbacks the stacks $\HiggsSt_{\widetilde{G}}^\bullet(C,L)$ for various choices of $x\in C$ are all isomorphic (although not canonically so).
\end{remark}

\begin{remark}\label{rk:canbun-conventions}
When $L=K_C$, the canonical bundle, I will often omit the line bundle from the notation, e.g.\
\begin{align}\label{eq:Higgs-KC-convention}
	\HiggsSt_G(C,K_C)\equiv\HiggsSt_G(C),\quad \Hitch_{\mf{g}}(C,K_C)=\Hitch_{\mf{g}}(C) ,\quad\text{etc.}
\end{align}
\end{remark}


\subsubsection{Local description over $\Hitch_{\mf{g}}(C)$ and definition of $\mc{M}_{\widetilde{G}}^\bullet(C)$}\label{sss:Hitchinfib}

Recall (Definition \ref{def:Hitchin-base}) that the Hitchin base is defined by $\Hitch_{\mf{g}}(C,L)= H^0(C;\mf{c}_L)$, so that for $\mf{g}_\tau$
\begin{align}\label{eq:Hitch-g-tau}
	\Hitch_{\mf{g}_\tau}(C,L)
		&= H^0(C;\mf{c}_L\times(\mf{t}\tens L))					\\
		&= H^0(C;\mf{c}_L)\times H^0(C;\mf{t}\tens L)	\nonumber\\
		&= \Hitch_{\mf{g}}(C,L)\times H^0(C;\mf{t}\tens L) . \nonumber
\end{align}

Restricting to the case $L=K_C$, the following square commutes (though is {\bf not} cartesian):
\begin{equation}\label{diag:Higgs-square}
\begin{tikzcd}
	\HiggsSt_{\widetilde{G}}^\bullet(C)\ar[r]\ar[d]
		& \HiggsSt_{\widetilde{G}_\tau}(C)\ar[d]	\\
	\Hitch_{\mf{g}}(C)\ar[r,"\id\times 0"']
		& \Hitch_{\mf{g}}(C)\times H^0(C;\mf{t}\tens K_C)
\end{tikzcd}
\end{equation}

\begin{remark}\label{rk:trace-free}
Since I wish to compare $\HiggsSt_{\widetilde{G}_\tau}(C)$ with $\HiggsSt_{\widetilde{G}}(C)$, from now on I will implicitly restrict $\HiggsSt_{\widetilde{G}_\tau}(C)$ to the trace-free locus $\Hitch_{\mf{g}}(C)\times\{0\}\subset\Hitch_{\mf{g}}\times H^0(C;\mf{t}\tens K_C)=\Hitch_{\mf{g}_\tau}(C)$.
\end{remark}

Note that
\begin{equation}\label{eq:centre-G-tau}
\begin{tikzcd}[row sep = tiny]
	Z(\widetilde{G}_\tau)\cong T \ar[r,hook]
		& \widetilde{G}_\tau = \frac{\widetilde{G}\times T}{Z(\widetilde{G})} 	\\
	t\ar[r,mapsto]
		& {[}(1_{\widetilde{G}},t){]}
\end{tikzcd}
\end{equation}
hence $\HiggsSt_{\widetilde{G}_\tau}(C)|_{\Hitch_{\mf{g}}(C)\setminus\Delta}\to\Higgs_{\widetilde{G}_\tau}(C)|_{\Hitch_{\mf{g}}(C)\setminus\Delta}$ is a $Z(\widetilde{G}_\tau)=T$-gerbe, locally trivial over $\Hitch_{\mf{g}}(C)\setminus\Delta$ \cite[Lemma 4.2.ii]{DonPan12}\footnote{While the lemma in \cite{DonPan12} is stated for semisimple groups, the proof of part (ii) holds \emph{mutatis mutandis} for our reductive group. The eagle eyed may observe a very minor typo in the proof -- namely that the centre of a group is contained in but usually not equal to the Weyl group invariants in a maximal torus -- but this typo does not affect the argument of \cite{DonPan12} nor the conclusion I wish to draw.}.

\begin{remark}[Important Remark!]\label{rk:good-locus}
From now on I will assume that we are working away from the discriminant locus \eqref{eq:discr}, and except for in the statement of theorems {\bf I will omit the explicit restriction symbol} ``$|_{\Hitch_{\mf{g}}(C)\setminus\Delta}$''.  Moreover, from now on whenever I refer to some property (e.g.\ a decomposition) of a space related to the moduli of Higgs bundles via a map commuting with the Hitchin map holding ``locally'' I mean that it holds \emph{locally on the Hitchin base, away from the discriminant locus}.
\end{remark}

In other words, locally the stack $\HiggsSt_{\widetilde{G}_\tau}(C)$ decomposes as the product
\begin{align}\label{eq:Higgs-G-tau-local-1}
	\HiggsSt_{\widetilde{G}_\tau}(C)\cong\Higgs_{\widetilde{G}_\tau}(C)\times BT .
\end{align}
Moreover, the coarse moduli space $\Higgs_{\widetilde{G}_\tau}(C)$ splits locally into the product of its neutral component and its group of connected components (since its group of connected components is the free group $\pi_0(\Higgs_{\widetilde{G}_\tau}(C))=\pi_0(\Bun_{T/Z(\widetilde{G})}(C))=X_\bullet(T/Z(\widetilde{G}))$), i.e.\ locally
\begin{align}\label{eq:Higgs-G-tau-local-2}
	\HiggsSt_{\widetilde{G}_\tau}(C) \simeq \Higgs^0_{\widetilde{G}_\tau}(C) \times X_\bullet(T/Z(\widetilde{G})) \times BT .
\end{align}

Next we wish to understand the local structure of $\HiggsSt^\bullet_{\widetilde{G}}(C)$. A (closed) point of $\HiggsSt_{\widetilde{G}}^\bullet(C)$ is given by
\begin{enumerate}[(1)]
	\item a $\widetilde{G}_\tau$-bundle $P\to C$
	\item a Higgs field $\phi\in H^0(C;\ad(P)\tens K_C)$ which is ``trace-free'', and
	\item an isomorphism $\psi: \pd_*(P)\simeq \mc{O}(\vec{p}x)$ (for some $\vec{p}\in\mb{Z}^s$).
\end{enumerate}
More generally, an $S$-point of $\HiggsSt_{\widetilde{G}}^\bullet(C)$ is given by
\begin{enumerate}[(1)]
	\item a $\widetilde{G}_\tau$-bundle $P_S\to C\times S$
	\item a Higgs field $\phi_S\in H^0(C\times S; \pr_C^*(\ad(P)\tens K_C))$ which is ``trace-free'', and
	\item an isomorphism $\psi_S : \pd_*(P_S)\simeq \pr_C^*(\mc{O}(\vec{p}x))$(for some locally constant function $\vec{p}:S\to\mb{Z}^s$).
\end{enumerate}

Note that the action of $BT$ which was previously given by tensoring with the pullback of a $T$-bundle on $S$ must be restricted: now only $T$-bundles $\mc{T}_S\to S$ satisfying
\begin{align}\label{eq:Z-tors-condition}
	\pd_*(\mc{T}_S)\simeq\mc{O}_S
\end{align}
may act on the moduli space. These are exactly those $T$-bundles which are induced from $Z(\widetilde{G})$-bundles via $\tau$,
\begin{equation}\label{eq:Btau-SES}
\begin{tikzcd}
	0\ar[r]	& BZ(\widetilde{G}) \ar[r,"B\tau"]	& BT\ar[r,"B\pd"]	& B(T/Z(\widetilde{G})) \ar[r]	& 0	,
\end{tikzcd}
\end{equation}
so we see that one effect of pulling back is a ``reduction of structure group'' from $BT$ to $BZ(\widetilde{G})$.

To see what happens to the abelian variety component in the local decomposition \eqref{eq:Higgs-G-tau-local-2}, note that the component defined by the cartesian diagram
\begin{equation}\label{diag:Higgs-sc-pullback}
\begin{tikzcd}
	\HiggsSt_{\widetilde{G}}^0(C,L)\ar[r]\ar[d]\arrow[dr, phantom,"\lrcorner", very near end]
		& \HiggsSt_{\widetilde{G}}^\bullet(C,L)\ar[d]	\\
	\ast\ar[r,"0"']
		& X_\bullet(T/Z(\widetilde{G}))
\end{tikzcd}
\end{equation}
may be identified as $\HiggsSt_{\widetilde{G}}^0(C,L)\simeq \HiggsSt_{\widetilde{G}}(C,L)$, the usual moduli of Higgs bundles for the simply-connected simple group $\widetilde{G}$. So locally $\HiggsSt_{\widetilde{G}}^\bullet(C)$ decomposes as
\begin{align}\label{eq:Higgs-bullet-local}
	\HiggsSt_{\widetilde{G}}^\bullet(C) \cong \Higgs_{\widetilde{G}}(C)\times X_\bullet(T/Z(\widetilde{G})) \times BZ(\widetilde{G}) .
\end{align}

The natural map $\HiggsSt_{\widetilde{G}}^\bullet(C)\to\HiggsSt_{\widetilde{G}_\tau}(C)$ is locally
\begin{equation}\label{eq:bullet-to-tau-local}
\begin{tikzcd}[column sep = tiny]
	\Higgs_{\widetilde{G}}(C) \ar[d,hook] & \times
		& X_\bullet(T/Z(\widetilde{G})) \ar[d,hook,"\id"] & \times 
			& BZ(\widetilde{G}) \ar[d,hook,"B\tau"]	\\
	\Higgs_{\widetilde{G}_\tau}^0(C) & \times
		&X_\bullet(T/Z(\widetilde{G})) &\times
			& BT
\end{tikzcd}
\end{equation}
and the projection $\HiggsSt_{\widetilde{G}}^\bullet(C)\to\HiggsSt_{G_{\ad}}(C)$ is locally
\begin{equation}\label{eq:bullet-to-ad-local}
\begin{tikzcd}[column sep = tiny]
	\Higgs_{\widetilde{G}}(C) \ar[d,"\text{isogeny}"'] & \times
		&X_\bullet(T/Z(\widetilde{G})) \ar[d] & \times 
			& BZ(\widetilde{G}) \ar[d]	\\
	\Higgs^0_{G_{\ad}}(C) & \times
		& Z(\widetilde{G}) &\times
			& \ast
\end{tikzcd}
\end{equation}

There is another important stack which admits a map from $\HiggsSt^\bullet_{\widetilde{G}}(C)$, constructed as follows. Since the multiplication map $T\times \widetilde{G}_\tau \to \widetilde{G}_\tau$ is a group homomorphism, it induces an action of $\BunSt_T(C)$ on $\HiggsSt_{\widetilde{G}_\tau}(C)$. Via the splitting $\iota_x: X_\bullet(T)\to\BunSt_T(C)$ we may restrict this to an action of $X_\bullet(T)$, which may be thought of concretely as tensoring Higgs bundles with $\mc{O}(\vec{p}x)$ as $\vec{p}$ ranges over $\mb{Z}^s$. This action evidently restricts to give an action of $X_\bullet(T)$ on $\HiggsSt^\bullet_{\widetilde{G}}(C)$.

\begin{defn}\label{def:M-moduli}
Denote by $\mc{M}_{\mf{g}}(C)$ the stack $\HiggsSt^\bullet_{\widetilde{G}}(C)/X_\bullet(T)$.
\end{defn}

\begin{prop}\label{prop:M-local}
The map $\HiggsSt^\bullet_{\widetilde{G}}(C)|_{\Hitch_{\mf{g}}(C)\setminus\Delta}\to\mc{M}_{\mf{g}}(C)|_{\Hitch_{\mf{g}}(C)\setminus\Delta}$ is locally given by
\begin{equation}\label{eq:bullet-to-M-local}
\begin{tikzcd}[column sep = tiny]
	\HiggsSt_{\widetilde{G}}^\bullet(C)\ar[d] & \simeq
		& \Higgs_{\widetilde{G}}(C) \ar[d,equals] & \times
			& X_\bullet(T/Z(\widetilde{G})) \ar[d] & \times 
				& BZ(\widetilde{G}) \ar[d,equals]	\\
	\mc{M}_{\mf{g}}(C) & \simeq
		& \Higgs_{\widetilde{G}}(C) & \times
			& Z(\widetilde{G}) &\times
				& BZ(\widetilde{G})
\end{tikzcd}
\end{equation}
\end{prop}

\begin{proof}
The action of $X_\bullet(T)$ on $\Higgs_{\widetilde{G}}(C)$ and $BZ(\widetilde{G})$ is trivial, so it suffices to check this claim for the group of connected components. For this, is suffices to check the corresponding claim for the moduli space of bundles (not Higgs bundles). Consider the generalisation of the Kummer sequence\footnote{This becomes the Kummer sequence for $T=\mb{G}_m$ and $Z(\widetilde{G})=\mu_n$.}
\begin{align}\label{eq:Kum-gen}
	1\to Z(\widetilde{G})\to T(\mc{O}_C)\to (T/Z(\widetilde{G}))(\mc{O}_C)\to 1 .
\end{align}

The $H^0$ row of the corresponding long exact sequence in cohomology is exact (since $C$ is compact/projective); starting at $H^1$ the long exact sequence is
\begin{equation}\label{eq:Kum-LES-1}
\begin{tikzcd}
	0\ar[r]
		& H^1(C;Z(\widetilde{G})) \arrow[r]
			& H^1(C;T(\mc{O}_C)) \arrow[r] \arrow[d, phantom, ""{coordinate, name=Z}]
				& H^1(C;(T/Z(\widetilde{G}))(\mc{O}_C)) \arrow[dll,rounded corners,to path={ -- ([xshift=2ex]\tikztostart.east)|- (Z) [near end]\tikztonodes-| ([xshift=-2ex]\tikztotarget.west)-- (\tikztotarget)}] \\
		& H^2(C;Z(\widetilde{G}))\arrow[r]
			& H^2(C;T(\mc{O}_C))
\end{tikzcd}
\end{equation}

Now, $H^2(C;T(\mc{O}_C))=0$ -- this follows analytically by taking the long exact sequence of the exponential sequence $0\to\mb{Z}\to\mc{O}\to\mc{O}^\times\to 1$ and observing that there are no (2,0)-forms on $C$, and it follows algebraically from the existence of an injective comparison map $H^2_{et}(C;\mb{G}_m)\to H^2(C^{an};\mc{O}_C^\times)$ \cite{DonPan08}.

Identifying $H^2(C;Z(\widetilde{G}))=Z(\widetilde{G})$ canonically and using the identification $H^1(C;T(\mc{O}_C))=\Bun_T(C)$, \eqref{eq:Kum-LES-1} becomes
\begin{align}\label{eq:Kum-LES-2}
	0\to H^1(C;Z(\widetilde{G}))\to\Bun_T(C)\to\Bun_{T/Z(\widetilde{G})}(C)\to Z(\widetilde{G})\to 0 .
\end{align}

The map out of $H^1(C;Z(\widetilde{G}))$ factors through the identity component of $\Bun_T(C)$, and so the content of \eqref{eq:Kum-LES-2} may be split into the two identifications: $\Bun_{T/Z(\widetilde{G})}^0(C)\cong \frac{\Bun_T^0(C)}{H^1(C;Z(\widetilde{G}))}$ and
\begin{align}\label{eq:ZG-components}
	Z(\widetilde{G})
		\cong \frac{\pi_0(\Bun_{T/Z(\widetilde{G})}(C))}{\pi_0(\Bun_T(C))}
		= \frac{X_\bullet(T/Z(\widetilde{G}))}{X_\bullet(T)} .
\end{align}
\end{proof}

\begin{eg}\label{eg:dmodn}
In the running example with $\widetilde{G}=SL_n$, $\mc{M}_{\mf{sl}_n}(C)$ may be thought of as encoding the observation that the moduli spaces $\Higgs_{SL_n}^d$ depend only on $d \mod n$, explicitly because tensoring with the line bundle $\mc{O}(x)$ is an isomorphism $\Higgs_{SL_n}^d \cong \Higgs_{SL_n}^{d+n}$.
\end{eg}


\subsection{Comparing sheaves of regular centralisers}\label{ss:compareJ}

In \eqref{eq:bullet-to-tau-local} we observed that $\Higgs_{\widetilde{G}}(C)$ appears as an abelian subvariety of $\Higgs_{\widetilde{G}_\tau}^0(C)$. Since under shifted Cartier duality (Definition \ref{def:Cartier-dual}) subobjects become quotient objects, and under Langlands duality simply-connected groups are sent to adjoint groups, it is natural to guess that $\Higgs_{G_{\ad}}^0(C)$ may be realised as a quotient of $\Higgs_{\widetilde{G}_\tau}^0(C)$. To see that this is indeed the case, we will compare the sheaves of regular centralisers $J^{\widetilde{G}_\tau}$ and $J^{\widetilde{G}}$.

\begin{prop}\label{prop:J-products}
\begin{enumerate}[(1)]
	\item[]
	\item $J^{G_1\times G_2} = J^{G_1}\times J^{G_2}$
	\item If $T\cong(\mb{C}^\times)^n$ then $J^T=T$.
	\item There is a short exact sequence of sheaves
	\begin{align}\label{eq:J-tau-SES}
		0\to Z(\widetilde{G}) \to J^{\widetilde{G}} \times T \to J^{\widetilde{G}_\tau}\to0 .
	\end{align}
\end{enumerate}
\end{prop}
\begin{proof}
\begin{enumerate}[(1)]
	\item[]
	\item Follows from the fact that the Lie algebra of $G_1\times G_2$ is $\mf{g}_1\oplus\mf{g}_2$, and the adjoint action factors as $G_1\times G_2\to\End(\mf{g}_1)\oplus\End(\mf{g}_2)\subset\End(\mf{g}_1\oplus\mf{g}_2)$.
	\item Since $T$ is abelian the adjoint action is trivial, so $Z_T(x)=T$ for every $x\in\mf{t}$.
	\item By (1) and (2) this is a special case of Lemma \ref{l:J-SES}.
\end{enumerate}
\end{proof}

\begin{prop}\label{prop:Higgs-ad-quotient}
There are isomorphisms of abelian schemes $\Higgs_{G_{\ad}}^0(C)|_{\Hitch_{\mf{g}}(C)\setminus\Delta}\cong\left.\frac{\Higgs_{\widetilde{G}_\tau}^0(C)}{\Bun_T^0(C)}\right|_{\Hitch_{\mf{g}}(C)\setminus\Delta}$ and $\left.\frac{\Bun_T^0(C)}{H^1(C;Z(\widetilde{G}))}\right|_{\Hitch_{\mf{g}}(C)\setminus\Delta}\cong\left.\frac{\Higgs_{\widetilde{G}_\tau}^0(C)}{\Higgs_{\widetilde{G}}(C)}\right|_{\Hitch_{\mf{g}}(C)\setminus\Delta}$.
\end{prop}

\begin{proof}
Pulling the short exact sequence \eqref{eq:J-tau-SES} back via some cameral cover $\tilde{C}_\sigma$ of $C$ yields
\begin{equation}\label{eq:J-tau-LES}
\begin{tikzcd}
	0\ar[r]
		& \Gamma(C;Z(\widetilde{G})) \arrow[r]
			& \Gamma(C;J^{\widetilde{G}}_\sigma\times T) \arrow[r] \arrow[d, phantom, ""{coordinate, name=Z}]
				& \Gamma(C;J^{\widetilde{G}_\tau}_\sigma) \arrow[dll,rounded corners,to path={ -- ([xshift=2ex]\tikztostart.east)|- (Z) [near end]\tikztonodes-| ([xshift=-2ex]\tikztotarget.west)-- (\tikztotarget)}] \\
		&H^1(C;Z(\widetilde{G})) \arrow[r]
			& H^1(C;J^{\widetilde{G}}_\sigma\times T) \arrow[r] \arrow[d, phantom, ""{coordinate, name=Y}]
				& H^1(C;J^{\widetilde{G}_\tau}_\sigma) \arrow[dll,rounded corners,to path={ -- ([xshift=2ex]\tikztostart.east)|- (Y) [near end]\tikztonodes-| ([xshift=-2ex]\tikztotarget.west)-- (\tikztotarget)}] \\
		&H^2(C;Z(\widetilde{G}))=Z(\widetilde{G}) \ar[r]
			& 0
\end{tikzcd}
\end{equation}
where the vanishing of $H^2(C;J^{\widetilde{G}}_\sigma\times T)$ is observed in \cite[\S5]{DonPan12}.

Let $K_\tau=\ker(H^1(C;J^{\widetilde{G}_\tau}_\sigma)\to H^2(C;Z(\widetilde{G})))$.\footnote{Note that this is not necessarily connected, i.e.\ is not necessarily the neutral component.} Then \eqref{eq:J-tau-LES} becomes
\begin{equation}\label{eq:K-tau-LES}
\begin{tikzcd}
	0\ar[r]
		& Z(\widetilde{G}) \arrow[r]
			& \Gamma(C;J^{\widetilde{G}}_\sigma)\times T \arrow[r] \arrow[d, phantom, ""{coordinate, name=Z}]
				& \Gamma(C;J^{\widetilde{G}_\tau}_\sigma) \arrow[dll,rounded corners,to path={ -- ([xshift=2ex]\tikztostart.east)|- (Z) [near end]\tikztonodes-| ([xshift=-2ex]\tikztotarget.west)-- (\tikztotarget)}] \\
		&H^1(C;Z(\widetilde{G})) \arrow[r]
			& H^1(C;J^{\widetilde{G}}_\sigma)\times H^1(C;T) \arrow[r]
				& K_\tau \ar[r]
					& 0 .
\end{tikzcd}
\end{equation}

Since the map $H^1(C;Z(\widetilde{G}))\to H^1(C;T)=\Bun_T(C)$ is itself an embedding (and in fact it factors through $H^1(C;T)^0 = \Bun_T^0(C)$), the above sequence splits into two short exact sequences, yielding
\begin{align}\label{eq:K-tau-quot}
	K_\tau = \frac{\Higgs_{\widetilde{G}}(C)\times\Bun_T(C)}{H^1(C;Z(\widetilde{G}))}
\end{align}
and (restricting the the neutral component)
\begin{align}\label{eq:Higgs-tau-0-quot}
	\Higgs_{\widetilde{G}_\tau}^0(C) = \frac{\Higgs_{\widetilde{G}}(C)\times\Bun^0_T(C)}{H^1(C;Z(\widetilde{G}))} .
\end{align}

The isomorphism $\frac{\Bun_T^0(C)}{H^1(C;Z(\widetilde{G}))}\cong\frac{\Higgs_{\widetilde{G}_\tau}^0(C)}{\Higgs_{\widetilde{G}}(C)}$ follows immediately from \eqref{eq:Higgs-tau-0-quot}, and the isomorphism $\Higgs_{G_{\ad}}^0(C)\cong\frac{\Higgs_{\widetilde{G}_\tau}^0(C)}{\Bun_T^0(C)}$ follows from \eqref{eq:Higgs-tau-0-quot} and the identification $\Higgs_{G_{\ad}}^0(C)\cong\frac{\Higgs_{\widetilde{G}}(C)}{H^1(C;Z(\widetilde{G}))}$ of Theorem \ref{thm:PrymComp}. To conclude the proof, note that the steps of the above local calculation can be replicated globally on $\textbf{H}^\circ:=\Hitch_{\mf{g}}(C)\setminus\Delta$. Namely, pullback the short exact sequence \eqref{eq:J-tau-SES} along the canonical evaluation map $\textbf{H}^\circ\times C\to\tot(\mf{c}_{K_C})$, and consider the long exact sequence of the derived pushforward along $\textbf{H}^\circ\times C\to \textbf{H}^\circ$. Since the embedding $H^1(C;Z(\widetilde{G}))\to\Bun_T(C)$ is insensitive to the Hitchin base, this gives a global identification $\Higgs_{\widetilde{G}_\tau}^0(C)|_{\textbf{H}^\circ}=\frac{\Higgs_{\widetilde{G}}(C)|_{\textbf{H}^\circ}\times\Bun^0_T(C)}{H^1(C;Z(\widetilde{G}))}$, and the result follows.
\end{proof}


\subsection{Dualising $\HiggsSt_{\widetilde{G}}^\bullet(C)$}\label{ss:bullet-dual}

At a first glance one might expect that the stacks $\HiggsSt_{\widetilde{G}}^\bullet(C)$ will provide the correct generalisation of the Langlands duality results of \cite{HauTha03,DonPan12}. In this section we will see that this is not quite correct, since by remembering all of the connected components of $\Higgs_{\widetilde{G}_\tau}(C)$ this stack is keeping track of too much information (or, perhaps better, it is keeping track of components and automorphisms in a non-symmetric manner). Regardless, I will describe the structure of the shifted Cartier dual $\HiggsSt_{\widetilde{G}}^\bullet(C)$ so that in Section \ref{ss:M-dual} I can show that the moduli space $\mc{M}_{\mf{g}}(C)$ \emph{is} well-behaved under shifted Cartier duality.

As a first step let us ``measure the difference'' between the stacks $\HiggsSt_{\widetilde{G}}^\bullet(C)$ and $\HiggsSt_{\widetilde{G}_\tau}(C)$, i.e.\ :

\begin{prop}\label{prop:measure-diff}
There are isomorphisms of commutative group stacks
\begin{align}\label{eq:Higgs-tau-BunTZ-quot}
	\HiggsSt_{\widetilde{G}_\tau}(C)/\HiggsSt_{\widetilde{G}}^\bullet(C)|_{\Hitch_{\mf{g}}(C)\setminus\Delta}
		\cong \left.\frac{\BunSt_{T/Z(\widetilde{G})}(C)}{X_\bullet(T/Z(\widetilde{G}))}\right|_{\Hitch_{\mf{g}}(C)\setminus\Delta}
		\cong\BunSt_{T/Z(\widetilde{G})}^0(C)|_{\Hitch_{\mf{g}}(C)\setminus\Delta}.
\end{align}
\end{prop}

\begin{proof}
The second isomorphism is immediate -- we have already seen that a choice of point $x\in C$ gives a splitting of the map $\BunSt_{T/Z(\widetilde{G})}(C)\to\pi_0(\BunSt_{T/Z(\widetilde{G})}(C))=X_\bullet(T/Z(\widetilde{G}))$. Hence it suffices to prove the first isomorphism, which follows by composing the pullback square \eqref{diag:Higgs-bullet-def} with the pullback square
\begin{equation}\label{diag:Opx-pullback}
\begin{tikzcd}
	X_\bullet(T/Z(\widetilde{G}))\ar[r]\ar[d]\ar[dr,phantom,"\lrcorner", very near end]
		& \BunSt_{T/Z(\widetilde{G})}(C)\ar[d]	\\
	0\ar[r]
		& \frac{\BunSt_{T/Z(\widetilde{G})}(C)}{X_\bullet(T/Z(\widetilde{G}))}
\end{tikzcd}
\end{equation}
to obtain the pullback square
\begin{equation}\label{diag:Higgs-bullet-fibre-seq}
\begin{tikzcd}
	\HiggsSt_{\widetilde{G}}^\bullet(C)\ar[r]\ar[d]\ar[dr,phantom,"\lrcorner", very near end]
		& \HiggsSt_{\widetilde{G}_\tau}(C)\ar[d]	\\
	0\ar[r]
		& \frac{\BunSt_{T/Z(\widetilde{G})}(C)}{X_\bullet(T/Z(\widetilde{G}))}
\end{tikzcd}
\end{equation}
This can be seen to yield a short exact sequence of commutative group stacks via the local description of the maps given in Section \ref{sss:Hitchinfib}.
\end{proof}

Now, consider the following short exact sequences of commutative group stacks and their coarse moduli spaces:
\begin{equation}\label{diag:coarse-moduli}
\begin{tikzcd}[row sep = tiny]
	0 \ar[r]
		& \HiggsSt_{\widetilde{G}}^\bullet(C)\ar[r]
			& \HiggsSt_{\widetilde{G}_\tau}(C)\ar[r]
				& \BunSt_{T/Z(\widetilde{G})}^0(C) \ar[r]
					& 0 				\\
	0 \ar[r]
		& \Higgs_{\widetilde{G}}^\bullet(C)\ar[r]
			& \Higgs_{\widetilde{G}_\tau}(C)\ar[r]
				& \Bun_{T/Z(\widetilde{G})}^0(C) \ar[r]
					& 0 				
\end{tikzcd}
\end{equation}
Using Example \ref{eg:BunT-dual} and the identifications given in Appendix \ref{s:structure-results} (as well as another dualisation result from \cite{DonPan12}, namely $^L\HiggsSt^0 = \Higgs^D $) these dualise to the short exact sequences
\begin{equation}\label{diag:coarse-moduli-dual}
\begin{tikzcd}[row sep = tiny]
	0 \ar[r]
		& \Bun_{^L(T/Z(\widetilde{G}))}(C)\ar[r]
			& \HiggsSt_{(\widetilde{^LG})_{^L\tau}}(C)\ar[r]
				& \HiggsSt^\bullet_{\widetilde{G}}(C)^D \ar[r]
					& 0 				\\
	0 \ar[r]
		& \Bun_{^L(T/Z(\widetilde{G}))}^0(C)\ar[r]
			& \HiggsSt_{(\widetilde{^LG})_{^L\tau}}^0(C)\ar[r]
				& \Higgs^\bullet_{\widetilde{G}}(C)^D\ar[r]
					& 0 
\end{tikzcd}
\end{equation}

From the exact sequences \eqref{diag:coarse-moduli-dual}, we are led to study the quotient stacks $\frac{\HiggsSt_{\widetilde{G}_\tau}(C)}{\Bun_T(C)}$ and $\frac{\HiggsSt^0_{\widetilde{G}_\tau}(C)}{\Bun^0_T(C)}$. By Proposition \ref{prop:Higgs-ad-quotient}, $\frac{\Higgs^0_{\widetilde{G}_\tau}(C)}{\Bun^0_T(C)}\cong\Higgs_{G_{\ad}}^0(C)$, so that $\frac{\HiggsSt^0_{\widetilde{G}_\tau}(C)}{\Bun^0_T(C)}$ is a $T$-gerbe over $\Higgs_{G_{\ad}}^0(C)$. This result extends to the non-neutral connected components as well:

\begin{prop}\label{p:quot-T-gerbe}
The stack $\left.\frac{\HiggsSt_{\widetilde{G}_\tau}(C)}{\Bun_T(C)}\right|_{\Hitch_{\mf{g}}(C)\setminus\Delta}$ is a $T$-gerbe over $\Higgs_{G_{\ad}}(C)|_{\Hitch_{\mf{g}}(C)\setminus\Delta}$.
\end{prop}

\begin{proof}
The exact sequence of groups
\begin{align}\label{eq:T-tau-ad-SES}
	1\to T \to \widetilde{G}_\tau \to G_{\ad} \to 1
\end{align}
yields the short exact sequence of sheaves of regular centralisers
\begin{align}\label{eq:JT-Jtau-Jad-SES}
	1\to T(\mc{O}_C) \to J_{\widetilde{G}_\tau}\to J_{G_{\ad}}\to 1 .
\end{align}
Global sections of \eqref{eq:JT-Jtau-Jad-SES} remain exact, so starting at $H^1$ the associated long exact sequence of cohomology gives
\begin{align}\label{eq:JT-Jtau-Jad-LES}
	0\to H^1(C;T(\mc{O}_C))\to H^1(C;J_{\widetilde{G}_\tau})\to H^1(C;J_{G_{\ad}})\to H^2(C;T(\mc{O}_C)) .
\end{align}
We have already seen that $H^2(C;T(\mc{O}_C))=0$ during the course of the proof of Proposition \ref{prop:M-local}, and so this becomes the short exact sequence of coarse moduli spaces
\begin{align}\label{eq:BunT-tau-ad-SES}
	0\to \Bun_T(C) \to \Higgs_{\widetilde{G}_\tau}(C)\to\Higgs_{G_{\ad}}(C)\to 0 .
\end{align}
Since $\HiggsSt_{\widetilde{G}_\tau}(C)$ is locally isomorphic to $\Higgs_{\widetilde{G}_\tau}(C)\times BT$ and admits a rigidification map to $\Higgs_{\widetilde{G}_\tau}(C)$ over all of $\Hitch_{\mf{g}}(C)\setminus\Delta$ that sends the $\Bun_T(C)$ of \eqref{diag:coarse-moduli-dual} to that of \eqref{eq:BunT-tau-ad-SES}, the result follows.
\end{proof}

Combining this result with the short exact sequences \eqref{diag:coarse-moduli-dual} gives the following corollary:

\begin{cor}\label{cor:dual-T-gerbe}
\begin{enumerate}[(a)]
	\item[]
	\item $\HiggsSt_{\widetilde{G}}^\bullet(C)^D|_{\Hitch_{\mf{g}}(C)\setminus\Delta}$ is an $^L(T/Z(\widetilde{G}))$-gerbe over $\Higgs_{^LG_{\ad}}(C)|_{\Hitch_{\mf{g}}(C)\setminus\Delta}$.
	\item $\Higgs_{\widetilde{G}}^\bullet(C)^D|_{\Hitch_{\mf{g}}(C)\setminus\Delta}$ is an $^L(T/Z(\widetilde{G}))$-gerbe over $\Higgs_{^LG_{\ad}}^0(C)|_{\Hitch_{\mf{g}}(C)\setminus\Delta}$.
\end{enumerate}
\end{cor}

\begin{notation}\label{not:Q-stack}
To declutter notation, from now on I will denote $\frac{\HiggsSt_{\widetilde{G}_\tau}(C)}{\Bun_T(C)}$ by $\mc{Q}_{\widetilde{G}}^\bullet(C)$.
\end{notation}


\subsection{Dualising $\mc{M}_{\mf{g}}(C)$}\label{ss:M-dual}

As per Example \ref{eg:dmodn}, the moduli stack $\mc{M}_{\mf{g}}(C)$ may be interpreted as the ``moduli of $\widetilde{G}$-Higgs bundles on $C$ of arbitrary degree, modulo uninteresting isomorphisms''. The main results of this paper -- namely the generalisation of \cite{HauTha03,DonPan12} to incorporate ``non-zero degrees'' for all semisimple groups (Theorems \ref{thm:dual-of-M} and \ref{thm:dual-of-intquot}) and the existence of self-dual moduli stacks associated to simply-laced Lie algebras (Corollary \ref{cor:self-dual-stacks}) -- boil down to the fact that the moduli stack $\mc{M}_{\mf{g}}(C)$ behaves nicely under shifted Cartier duality.

There is an action of $H^1(C;Z(\widetilde{G}))$ on $\mc{M}_{\mf{g}}(C)$, induced by the $\BunSt_T(C)$ action on $\HiggsSt_{\widetilde{G}_\tau}(C)$ and the trivialisation of the gerbe $\BunSt_T(C)$ over $\Bun_T(C)$ given by the choice of point $x\in C$.\footnote{\label{fn:pi0-dual-rigid}The existence of such a trivialisation  may be easier to see from the Cartier dual perspective, where it becomes the splitting of the map $\BunSt_{^LT}(C)\to \pi_0(\BunSt_{^LT}(C))=X_\bullet(^LT)$.} This action is free away from the discriminant locus of $\Hitch_{\mf{g}}(C)$, a fact which may be checked locally.

\begin{thm}\label{thm:dual-of-M}
There is an isomorphism of commutative group stacks
\begin{align}\label{eq:M-mod-H1-dual}
	\left(\left.\frac{\mc{M}_{\mf{g}}(C)}{H^1(C;Z(\widetilde{G}))}\right|_{\Hitch_{\mf{g}}(C)\setminus\Delta}\right)^D \cong \mc{M}_{^L\mf{g}}(C)|_{\Hitch_{^L\mf{g}}(C)\setminus\Delta} .
\end{align}
\end{thm}

\begin{remark}\label{rk:identify-hitchin-bases}
In the statement of Theorem \ref{thm:dual-of-M} and in what follows, I will implicitly identify the Hitchin bases for dual Lie algebras via the duality isomorphism of \cite[Theorem A]{DonPan12}.
\end{remark}

\begin{proof}
Consider the commutative diagram with exact rows and columns, where the first row of vertical arrows is induced by the trivialisation of $\BunSt_T(C)$:\footnote{In particular, the vertical map out of $\Bun_T(C)$ factors through $\BunSt_T(C)$.}
\begin{equation}\label{diag:Higgs-tau-diagram}
\begin{tikzcd}
		& 0\ar[d]
			& 0\ar[d]
				& 0\ar[d]
					&	\\
	0\ar[r]
		& X_\bullet(T)\times H^1(C;Z(\widetilde{G}))\ar[r]\ar[d]
			& \Bun_T(C)\ar[r]\ar[d]
				& \Bun_{T/Z(\widetilde{G})}^0(C)\ar[r]\ar[d]
					& 0 \\
	0\ar[r]
		& \HiggsSt_{\widetilde{G}}^\bullet(C)\ar[r]\ar[d]
			& \HiggsSt_{\widetilde{G}_\tau}(C)\ar[r]\ar[d]
				& \BunSt_{T/Z(\widetilde{G})}^0(C)\ar[r]\ar[d]
					& 0 \\
	0\ar[r]
		& \frac{\mc{M}_{\mf{g}}(C)}{H^1(C;Z(\widetilde{G}))}\ar[r]\ar[d]
			& \mc{Q}_{\widetilde{G}}^\bullet(C)\ar[r]\ar[d]
				& B(T/Z(\widetilde{G}))\ar[r]\ar[d]
					& 0 \\
		& 0 & 0 & 0 &
\end{tikzcd}
\end{equation}
Implicit in the above diagram is the claim that the map $\Bun_T(C)\to\HiggsSt_{\widetilde{G}_\tau}(C)$ induced by the inclusion $T\subset J^{\widetilde{G}_\tau}$ can be identified with the map out of $\Bun$ of \eqref{diag:coarse-moduli-dual}. This claim is non-obvious, and can be established by comparing the maps locally as follows.

Using the notation of \eqref{diag:coarse-moduli} and \eqref{diag:coarse-moduli-dual}, consider the maps $\pd$ and $\zeta$ defined by
\begin{equation}
\begin{tikzcd}
	1\ar[r] & \widetilde{G}\ar[r] & \widetilde{G}_\tau\ar[r,"\pd"] & T/Z(\widetilde{G}) \ar[r] & 1
\end{tikzcd} \label{eq:del-defn}
\end{equation}
\begin{equation}
\begin{tikzcd}
1 \ar[r] & {^L}(T/Z(\widetilde{G}))\ar[r,"\zeta"] & \left(\widetilde{^LG}\right)_{^L\tau}\ar[r] & {^LG}_{\ad}\ar[r] & 1
\end{tikzcd}\label{eq:zeta-defn}
\end{equation}
These induce maps
\[\begin{tikzcd}[row sep = tiny]
	\HiggsSt_{\widetilde{G}_\tau}(C)\ar[r,"\pd_\ast"]
		& \BunSt_{T/Z(\widetilde{G})}(C) \\
	\BunSt_{^L(T/Z(\widetilde{G}))}(C)\ar[r,"\zeta_\ast"]
		& \HiggsSt_{\left(\widetilde{^LG}\right)_{^L\tau}}(C)
\end{tikzcd}\]
By definition, $\HiggsSt_{\widetilde{G}_\tau}(C)\to\BunSt^0_{T/Z(\widetilde{G})}(C)$ factors through $\pd_\ast$, and $\Bun_{^L(T/Z(\widetilde{G}))}(C)\to\HiggsSt_{\left(\widetilde{^LG}\right)_{^L\tau}}(C)$ factors through $\zeta_\ast$. Moreover, the rigidification map is defined to be Cartier dual to splitting off $\pi_0$ (Footnote \ref{fn:pi0-dual-rigid}). Thus it suffices to prove that $\pd_\ast^D = \zeta_\ast$.

Over a point $\sigma\in\Hitch_{\mf{g}}(C)\setminus\Delta$,  the maps decompose as
\begin{align}
	\pd_\ast&:
		\underbrace{\Higgs^0_{\widetilde{G}_\tau}(C)_\sigma\times X_\bullet(T/Z(\widetilde{G}))\times BT}_{\simeq\HiggsSt_{\widetilde{G}_\tau}(C)_\sigma}
			\to\underbrace{\Bun_{T/Z(\widetilde{G})}^0(C)\times X_\bullet(T/Z(\widetilde{G}))\times B(T/Z(\widetilde{G}))}_{\simeq\BunSt_{T/Z(\widetilde{G})}(C)} \label{eq:pd-ast-decomp} \\
	\zeta_\ast&:
		\underbrace{\Bun^0_{^L(T/Z(\widetilde{G}))}(C)\times X^\bullet(T/Z(\widetilde{G}))\times B({^L(T/Z(\widetilde{G}))})}_{\simeq \BunSt_{^L(T/Z(\widetilde{G}))}(C)}
			\to
			 \underbrace{\Higgs^0_{\left(\widetilde{^LG}\right)_{^L\tau}}(C)_\sigma\times X^\bullet(T)\times B({^L(T/Z(\widetilde{G}))})}_{\HiggsSt_{\left(\widetilde{^LG}\right)_{^L\tau}}(C)_\sigma} \label{eq:zeta-ast-decomp}
\end{align}
with no ``mixing'' between distinct factors. We can deal with the three factors separately:
\begin{itemize}
	\item {\bf $\pd_\ast$ on $\pi_0$:} $\pd_\ast$ restricts to the identity map on $\pi_0$, which dualises correctly to the identity map on $B({^L(T/Z(\widetilde{G}))})$.

	\item {\bf $\pd_\ast$ on $BT$:} $\pd_\ast$ sends a $T$-bundle $U\to S$ to the $T/Z(\widetilde{G})$-bundle $U/Z(\widetilde{G})\to S$; i.e.\ it is $Bq:BT\to B(T/Z(\widetilde{G}))$ where $q:T\to T/Z(\widetilde{G})$ is the quotient map. This dualises to the pullback of characters $q^\ast: X^\bullet(T/Z(\widetilde{G}))\to X^\bullet(T)$.

	On the other hand, the map $\HiggsSt_{\left(\widetilde{^LG}\right)_{^L\tau}}(C)_\sigma \to \BunSt_{^LT}(C)$ is induced by the map
	\[	1\to \widetilde{^LG}
			\to \left(\widetilde{^LG}\right)_{^L\tau}
				= \frac{\widetilde{^LG}\times{^L(T/Z(\widetilde{G}))}}{Z(\widetilde{^LG})}
				\to {^LT} \to 1	\]
	and so the composition with $\zeta$ is the quotient map $^L(T/Z(\widetilde{G}))\to {^LT}$. This quotient map is equal to $q^\ast\tens\mb{C}^\times$ (see \eqref{eq:Ltau-SES})). The induced map on connected components is given by taking cocharacters, which yields $q^\ast = \zeta_\ast$.

	\item {\bf $\pd_\ast$ on $\Higgs^0_{\widetilde{G}_\tau}(C)_\sigma$:} Using the cohomological description of the Hitchin fibres and \eqref{eq:Higgs-tau-0-quot}, we have that the restriction of $\pd_\ast$ and $\zeta_\ast$ to the abelian variety factors fit into factorisation diagrams
	\begin{equation}\label{diag:factor-pd-ast}
	\begin{tikzcd}
		0\to
			 H^1(C;Z(\widetilde{G})) \ar[r]
				& \frac{H^1(C;J^{\widetilde{G}}_\sigma)\times H^1(C;T)_0}{H^1(C;Z(\widetilde{G}))} \ar[d,"\pd_\ast"]\ar[r,"\text{isog}_1"]
					& H^1(C;J^{G_{\ad}}_\sigma)\times H^1(C;T/Z(\widetilde{G}))_0 \ar[dl,"\widetilde{\pd_\ast}"]
						\to 0 \\
				& H^1(C;T/Z(\widetilde{G}))_0
	\end{tikzcd}
	\end{equation}
	\begin{equation}\label{diag:factor-zeta-ast}
	\begin{tikzcd}
		0\to
			 H^1(C;Z(\widetilde{^LG})) \ar[r]
				& H^1(C;J^{\widetilde{^LG}}_\sigma)\times H^1(C;{^L(T/Z(\widetilde{G}))})_0\ar[r,"\text{isog}_2"]
					& \frac{H^1(C;J^{\widetilde{^LG}}_\sigma)\times H^1(C;{^L(T/Z(\widetilde{G}))})_0}{H^1(C;Z(\widetilde{^LG}))}
						\to 0 \\
				& H^1(C;{^L(T/Z(\widetilde{G}))})_0 \ar[u,"\widetilde{\zeta_\ast}"]\ar[ur,"\zeta_\ast"]
	\end{tikzcd}
	\end{equation}
	where the subscript $0$ indicates the connected component of the identity. The rows are dual isogenies of abelian varieties, and the the inclusion $\widetilde{\zeta_\ast}$ is dual to the projection $\widetilde{\pd_\ast}$. Hence
		\[	\zeta_\ast
			=\text{isog}_2\circ\widetilde{\zeta_\ast}
			=\text{isog}_1^D \circ \widetilde{\pd_\ast}^D
			=(\widetilde{\pd_\ast}\circ\text{isog}_1)^D
			=\pd_\ast^D,	\]
	which completes the proof of the claim.
\end{itemize}

With that dealt with, we now dualise the bottom row of diagram \eqref{diag:Higgs-tau-diagram} to obtain
\begin{equation}\label{diag:Higgs-tau-dual-diagram}
\begin{tikzcd}
	0\ar[r]
		& X_\bullet({^L(T/Z(\widetilde{G}))})\ar[r]
			& \mc{Q}_{\widetilde{G}}^\bullet(C)^D\ar[r]
				& \left(\frac{\mc{M}_{\mf{g}}(C)}{H^1(C;Z(\widetilde{G}))}\right)^D\ar[r]
					& 0 .
\end{tikzcd}
\end{equation}
Then by the definition of $\mc{Q}_{\widetilde{G}}^\bullet(C)$ and Proposition \ref{p:quot-T-gerbe}, $\mc{Q}_{\widetilde{G}}^\bullet(C)^D \cong \HiggsSt_{\widetilde{^LG}}^\bullet(C)$, so we conclude that
\begin{align}\label{eq:M-mod-H1-dual-redux}
	\left(\frac{\mc{M}_{\mf{g}}(C)}{H^1(C;Z(\widetilde{G}))}\right)^D \cong \frac{\HiggsSt_{\widetilde{^LG}}^\bullet(C)}{X_\bullet({^L(T/Z(\widetilde{G}))})} =: \mc{M}_{^L\mf{g}}(C) .
\end{align}
\end{proof}

Now, take a subgroup $\Gamma\subset H^1(C;Z(\widetilde{G}))$ and consider the ``intermediate quotient'' stack $\frac{\mc{M}_{\mf{g}}(C)}{\Gamma}$. $H^1(C;Z(\widetilde{G}))$ is equipped with a non-degenerate skew pairing, induced by the cup product on cohomology and a natural nondegenerate symmetric pairing on $Z(\widetilde{G})$ (see, e.g., \cite[\S 7.1]{KapWit07}). Denote by $\ann(\Gamma)$ the annihilator of $\Gamma$ with respect to this pairing.

\begin{thm}\label{thm:dual-of-intquot}
There is an isomorphism of commutative group stacks
\begin{align}\label{eq:dual-of-intquot}
	\left(\left.\frac{\mc{M}_{\mf{g}}(C)}{\Gamma}\right|_{\Hitch_{\mf{g}}(C)\setminus\Delta}\right)^D
		\cong\left.\frac{\mc{M}_{^L\mf{g}}(C)}{\ann(\Gamma)}\right|_{\Hitch_{^L\mf{g}}(C)\setminus\Delta} .
\end{align}
\end{thm}

\begin{proof}
Consider the quotient map
\begin{align}\label{eq:gamma-quot}
	\gamma:\frac{\mc{M}_{\mf{g}}(C)}{\Gamma} \to \frac{\mc{M}_{\mf{g}}(C)}{H^1(C;Z(\widetilde{G}))}
\end{align}
with kernel $H^1(C;Z(\widetilde{G}))/\Gamma$. Locally the map \eqref{eq:gamma-quot} is
\begin{equation}\label{eq:gamma-local}
\begin{tikzcd}[column sep = tiny]
\frac{\mc{M}_{\mf{g}}(C)}{\Gamma}\ar[d,"\gamma"']	
	& \simeq &\frac{\Higgs_{\widetilde{G}}(C)}{\Gamma}\ar[d,"\text{isogeny}"']
		& \times & Z(\widetilde{G})\ar[d,"\id_{Z(\widetilde{G})}"]
		& \times & BZ(\widetilde{G})\ar[d,"\id_{BZ(\widetilde{G})}"] 	\\
\frac{\mc{M}_{\mf{g}}(C)}{H^1(C;Z(\widetilde{G}))}
	& \simeq &\frac{\Higgs_{\widetilde{G}}(C)}{H^1(C;Z(\widetilde{G}))}
		& \times & Z(\widetilde{G})
		& \times & BZ(\widetilde{G})
\end{tikzcd}
\end{equation}
Under Cartier duality $(-)^D$, the map $\gamma$ dualises locally to
\begin{equation}\label{eq:gamma-dual-local}
\begin{tikzcd}[column sep = tiny]
\left(\frac{\mc{M}_{\mf{g}}(C)}{H^1(C;Z(\widetilde{G}))}\right)^D\ar[d,"\gamma^D"']	
	& \simeq &\left(\frac{\Higgs_{\widetilde{G}}(C)}{H^1(C;Z(\widetilde{G}))}\right)^D\ar[d,"\text{dual isogeny}"']
		& \times & BZ(\widetilde{^LG})\ar[d,"\id_{BZ(\widetilde{^LG})}"]
		& \times & Z(\widetilde{^LG})\ar[d,"\id_{Z(\widetilde{^LG})}"] 	\\
\left(\frac{\mc{M}_{\mf{g}}(C)}{\Gamma}\right)^D
	& \simeq &\left(\frac{\Higgs_{\widetilde{G}}(C)}{\Gamma}\right)^D
		& \times & BZ(\widetilde{^LG})
		& \times & Z(\widetilde{^LG})
\end{tikzcd}
\end{equation}
The kernel of the dual isogeny is $(H^1(C;Z(\widetilde{G}))/\Gamma)^\vee$, so we have a short exact sequence
\begin{align}\label{eq:intquot-dual-SES}
	0\to (H^1(C;Z(\widetilde{G}))/\Gamma)^\vee
			\to\left(\frac{\mc{M}_{\mf{g}}(C)}{H^1(C;Z(\widetilde{G}))}\right)^D
			\to\left(\frac{\mc{M}_{\mf{g}}(C)}{\Gamma}\right)^D
			\to 0 .
\end{align}
The theorem now follows from the identification $\left(\frac{\mc{M}_{\mf{g}}(C)}{H^1(C;Z(\widetilde{G}))}\right)^D \cong \mc{M}_{^L\mf{g}}(C)$ of Theorem \ref{thm:dual-of-M}, and the identification $(H^1(C;Z(\widetilde{G}))/\Gamma)^\vee \cong \ann(\Gamma)$ induced by the non-degenerate skew-pairing.
\end{proof}

\begin{remark}\label{rk:syz-1}
By restricting to the semistable locus and letting $\Gamma$ be a subgroup induced by a subgroup $Z\subset Z(\widetilde{G})$, Theorem \ref{thm:dual-of-intquot} may be interpreted as an SYZ mirror symmetry statement relating Hitchin fibrations for arbitrary semisimple Langlands dual groups coupled to nontrivial finite $B$-fields \cite{SYZ96,HauTha03}.

In fact, in type A it is possible to derive from this a topological mirror symmetry statement in the vein of \cite{HauTha03} by applying the results of \cite{GWZ17}. In that paper, Groechenig, Wyss and Ziegler prove an equality of ``gerbe-twisted stringy $E$-polynomials'' -- roughly speaking, these record appropriately defined Hodge numbers for complex varieties with at worst orbifold singularities, equipped with a finite group gerbe\footnote{There are additionally some (fairly harmless) arithmetic conditions, since \cite{GWZ17} uses techniques that require the moduli spaces to be defined over rings more general than $\bC$.} -- for the spaces $\Higgs_{SL_n}^d(C)$ and $\Higgs_{PGL_n}^e(C)$ where $d$ and $e$ are both coprime to $n$. In the setup we have been considering, we can say the following:
\begin{itemize}
	\item {\bf Type A:} The equality of stringy E-polynomials will hold for $\frac{\Higgs_{SL_n}^d(C)}{\Gamma}$ and $\frac{\Higgs_{SL_n}^e(C)}{\ann(\Gamma)}$, respectively equipped with the $e^{th}$ and $d^{th}$ powers of the gerbe of liftings, when $\gcd(d,n)=\gcd(e,n)=1$. When $\Gamma$ is isotropic the arguments of \cite{GWZ17} apply directly and the only thing to check is that the isogeny $\frac{\Higgs_{SL_n}^d(C)}{\Gamma} \to \frac{\Higgs_{SL_n}^e(C)}{\ann(\Gamma)}$ is self-dual, which follows from the natural Pontrjagin self-duality of $\frac{\ann(\Gamma)}{\Gamma}$. When $\Gamma$ is not isotropic one may apply the results of the sequel \cite{GWZ18} where the same authors show that the hypothesis that the isogeny is self-dual may be weakened; the relevant isogenies between Hitchin Pryms in this case were constructed by Ng\^o\footnote{I thank Michael Groechenig for directing me to this construction.} \cite[4.18.1]{Ngo10}.
	\item {\bf Outside of Type A:} We hit a serious snag here, in that the coprimality assumption that ensured the existence of smooth components of $\mbf{M}_{\mf{g}}(C)$ in the Type A setup no longer applies. Worse, one can in fact guarantee the existence of strictly semistable points in \emph{every} connected component \cite{Rama75}!

	Nevertheless, away from the singular locus we have all of the ingredients that we want -- e.g.\ the arithmetic gerbe and Ng\^o's isogeny between dual Hitchin Pryms -- and so one might still hope to obtain a topological mirror symmetry statement either by extending the results of \cite{GWZ17,GWZ18} to allow for some singular behaviour, or by varying the stability condition and studying a non-singular birational model as in \cite{ChaLau10}.
\end{itemize}
\end{remark}

We may deduce from Theorem \ref{thm:dual-of-intquot} the existence of a collection of self-dual commutative group stacks:

\begin{cor}\label{cor:self-dual-stacks}
In the setup of Theorem \ref{thm:dual-of-intquot} suppose that $\widetilde{G}=\widetilde{^LG}$ (e.g.\ $\widetilde{G}$ is ADE type), and that $\Gamma = \ann(\Gamma)$ is a Lagrangian subgroup of $H^1(C;Z(\widetilde{G}))$. Then,
\begin{align}\label{eq:self-dual-M}
	\left(\left.\frac{\mc{M}_{\mf{g}}(C)}{\Gamma}\right|_{\Hitch_{\mf{g}}(C)\setminus\Delta}\right)^D
		\cong \left.\frac{\mc{M}_{\mf{g}}(C)}{\Gamma}\right|_{\Hitch_{\mf{g}}(C)\setminus\Delta},
\end{align}
i.e.\ $\left.\frac{\mc{M}_{\mf{g}}(C)}{\Gamma}\right|_{\Hitch_{\mf{g}}(C)\setminus\Delta}$ is a self-dual commutative group stack.
\end{cor}

\begin{remark}\label{rk:syz-2}
As per Remark \ref{rk:syz-1}, Corollary \ref{cor:self-dual-stacks} may be interpreted as the statement that a particular space is \emph{self} SYZ mirror dual. Combined with the consistency of $\frac{\mc{M}_{\mf{g}}(C)}{\Gamma}$ with physical expectations \cite{GMN-BPS,Tachi14}, it is therefore reasonable to conjecture that this space is the 3d Coulomb branch for a theory of class $\mc{S}$.
\end{remark}

Finally, we may deduce from the above results the following (non-stacky) corollary:

\begin{cor}\label{cor:duality-thm}
With notation as above, $\left.\frac{\Higgs_{\tilde{G}}(C)}{\Gamma}\right|_{\Hitch_{\mf{g}}(C)\setminus\Delta}$ and $\left.\frac{\Higgs_{\widetilde{^LG}}(C)}{\ann(\Gamma)}\right|_{\Hitch_{^L\mf{g}}(C)\setminus\Delta}$ are torsors for dual abelian schemes. In particular, if $\widetilde{G} = \widetilde{^LG}$ and $\ann(\Gamma)=\Gamma$ then $\left.\frac{\Higgs_{\tilde{G}}(C)}{\Gamma}\right|_{\Hitch_{\mf{g}}(C)\setminus\Delta}$ is a torsor for a self-dual abelian scheme.
\end{cor}
\begin{proof}
This follows from the previous results by restricting to the neutral component of the coarse moduli space.
\end{proof}


\subsection{Equivalence of derived categories}\label{ss:derived-equivalence}

Let us conclude this section by noting the implications for the derived categories of the moduli stacks we have been studying. Throughout this subsection we always work away from the discriminant locus in the Hitchin base.

Recall that given dual abelian schemes $X$ and $Y=X^D$ over a base $B$ with sheaves of sections $\mc{X}$ and $\mc{Y}$, an argument from the Leray spectral sequence implies that
\begin{align}\label{eq:X-torsors}
	H^1(B;\mc{X})=\left\{\begin{array}{c} \text{equivalence classes of} \\ \text{$X$-torsors over $B$} \end{array}\right\}
	\simeq H^2(Y;\mc{O}^\times) =
	\left\{\begin{array}{c} \text{equivalence classes of} \\ \text{$\mc{O}^\times$-gerbes over $Y$} \end{array}\right\}
\end{align}
provided that the local system $R^2\pi_{Y,\ast}\mc{O}^\times$ has no sections (generically true) and the pullback map $H^2(B;\mc{O}^\times)\to H^2(Y;\mc{O}^\times)$ is trivial; furthermore, if certain compatibility conditions are met then given $\beta\in H^1(B;\mc{X})$ and $\alpha\in H^1(B;\mc{Y})$ we may construct an $\mc{O}^\times$-gerbe determined by $\alpha$ over the $X$-torsor labelled by $\beta$ \cite{B-B09,DonPan08}. Call this gerbe $_\alpha X_\beta$.

\begin{cor}\label{cor:derived-cats}
Over $\Hitch_{\mf{g}}(C)\setminus\Delta$ there is an equivalence of bounded derived categories of coherent sheaves
\begin{align}\label{eq:full-stack-equiv}
	D^b_c\left( \frac{\mc{M}_{\mf{g}}(C)}{\Gamma} \right)
		\simeq
	D^b_c\left( \frac{\mc{M}_{^L\mf{g}}(C)}{\ann(\Gamma)} \right)
\end{align}
implemented by a Fourier-Mukai transform. Furthermore, for every $\beta\in\pi_0\left(\frac{\mc{M}_{\mf{g}}(C)}{\Gamma}\right)=\pi_1(G_{\ad})= Z(\widetilde{^LG})^\vee$ and $\alpha\in\pi_0\left(\frac{\mc{M}_{^L\mf{g}}(C)}{\ann(\Gamma)}\right)=\pi_1(^LG_{\ad})= Z(\widetilde{G})^\vee$ the equivalence \eqref{eq:full-stack-equiv} induces a Fourier-Mukai equivalence
\begin{align}\label{eq:componentwise-equiv}
	D^b_c\left( {_\alpha(\mathbf{M}_{\mf{g}}(C)/\Gamma)_\beta} \right)
		\simeq
	D^b_c\left( {_{-\beta}(\mathbf{M}_{^L\mf{g}}(C)/\ann(\Gamma))_\alpha} \right)
\end{align}
between the derived categories of weight 1 sheaves on the induced $\mc{O}^\times$-gerbes.
\end{cor}
\begin{proof}\label{proof:derived-cats}
Since the stacks involved are reflexive, \eqref{eq:full-stack-equiv} follows immediately from \eqref{eq:dual-of-intquot} (see e.g.\ \cite[Appendix]{DonPan08}). The compatibility conditions of \cite{B-B09} are satisfied since we may explicitly construct the desired $\mc{O}^\times$-gerbes as induced from the corresponding finite group gerbes via $\alpha: Z(\widetilde{G})\to\mb{C}^\times\subset\mc{O}^\times$, and similarly for $\beta$. Then as in \cite{DonPan12} we may apply the results of \cite[\S5-6]{B-B09} to obtain the statement \eqref{eq:componentwise-equiv} (cf.\ especially \cite[Corollary 6.2]{B-B09}).
\end{proof}

\begin{remark}\label{rk:sheaves-on-gerbes}
Recalling that the category of sheaves on an $\mc{O}^\times$ gerbe is $\mb{Z}$-graded, the minus sign $(\alpha,\beta)\to(-\beta,\alpha)$ that appears in \eqref{eq:componentwise-equiv} may be reinterpreted as saying that the Fourier-Mukai transform induces an equivalence between
\begin{align}\label{eq:Db-weights}
	D^b_c({_\alpha X_\beta},+1) \simeq D^b_c({_\beta Y_\alpha},-1),
\end{align}
where the $\pm 1$ denote the weight $\pm 1$ components of the corresponding derived categories. The change in weight arises from a simple analysis of how one expects the Fourier-Mukai transform to act on an $\mc{O}^\times$ gerbe over a torsor, which I have included as Appendix \ref{s:FM}. 
\end{remark}

\label{c}

%% file: section-examples.tex

\section{Examples of dual spaces}\label{s:many-egs}

To conclude, let us see how the results of Section \ref{s:stack-dual} may be used to both describe new dualities and reinterpret some previously known examples. Important Remark \ref{rk:good-locus} still applies -- i.e.\ all moduli stacks in this section are defined over the complement of the discriminant locus $\Delta\subset\Hitch_{\mf{g}}(C)$.

\begin{eg}\label{eg:A1theory-ABcycles}
An analysis of $A_1$ theories of class $\mc{S}$ was performed in \cite{GMN-BPS}. There Gaiotto, Moore and Neitzke explain that a line operator in the $A_1$ theory corresponds to a simple closed path on $C$, and that a collection of line operators may be simultaneously included in the theory only if a ``mutual locality condition'' is satisfied. Geometrically, the mutual locality condition on a collection of line operators $\ms{L}$ becomes the requirement that that the number of intersection points of any two paths in $\ms{L}$ be even -- by passing to Poincar\'e dual cocycles, this induces an isotropic subgroup of $H^1(C;\mu_2)$ with respect to the natural skew-pairing.

A well-defined $A_1$ theory requires a choice of a maximal collection of mutually local line operators, which induces a Lagrangian subgroup $\Gamma\subset H^1(C;\mu_2)$. \cite{GMN-BPS} propose that the resulting moduli space ought not to be $\Higgs_{SL_2}(C)$, but instead should be $\Higgs_{SL_2}(C)/\Gamma$. Corollary \ref{cor:duality-thm} tells us that this space is indeed self-dual, while Corollary \ref{cor:self-dual-stacks} suggests that if we wish to consider Higgs fields on topologically non-trivial bundles then we will have to account for some stacky structure in the form of a 2-form $B$-field \cite{HauTha03}.
\end{eg}

\begin{eg}\label{eg:HauTha-results}
Theorem \ref{thm:dual-of-M} in fact gives another derivation of the SYZ mirror symmetry results of Hausel and Thaddeus for $SL/PGL$-Higgs bundles \cite{HauTha03}. To see this, observe that for type $A_{n-1}$ \eqref{eq:M-mod-H1-dual} becomes
\begin{align}\label{eq:type-A-dual}
	\left( \frac{\mc{M}_{\mf{sl}_n\mb{C}}(C)}{H^1(C;\mb{Z}/n\mb{Z})}\right)^D
		\cong \mc{M}_{\mf{sl}_n\mb{C}}(C).
\end{align}
The right hand side of this equation is the moduli stack of $GL_n\mb{C}$-Higgs bundles $(E,\phi)$ equipped with an isomorphism $\det(E)\simeq\mc{O}_C(dx)$ for some degree $d\in\mb{Z}/n\mb{Z}$\footnote{The dependence on $d$ mod $n$ rather than $d\in\mb{Z}$ is also observed in \cite{HauTha03}.} and such that $\tr\phi =0$, and the object we are dualising on the left hand side is the moduli space of $PGL_n\mb{C}$-Higgs bundles equipped with the gerbe of liftings of the universal projective Higgs bundle to a universal $GL_n$-Higgs bundle (again, tracefree and equipped with an isomorphism $\det(E)\simeq\mc{O}_C(dx)$). The exact form of \cite[Thm.\ 3.7]{HauTha03} for $d,e\in\mb{Z}/n\mb{Z}$ then resembles \eqref{eq:componentwise-equiv}:
\begin{align}\label{eq:componentwise-typeA}
	D^b_c\left( {_d \mathbf{M}_{\mf{sl}_n}(C)_e} \right)
		\simeq
	D^b_c\left( {_{-e}(\mathbf{M}_{\mf{sl}_n}(C)/H^1(C;\mb{Z}/n\mb{Z}))_d} \right).
\end{align}
\end{eg}

\begin{eg}\label{eq:SO2n}
Consider the group $G=SO(2n)$. This is a self Langlands dual group, and so by the results of Donagi and Pantev \cite{DonPan12} gives rise to a self-dual moduli stack of Higgs bundles. It is natural to ask whether or not this space fits into the story of this paper.

In fact it does: for simplicity I will discuss this duality on the level of coarse moduli spaces. The centre of the universal cover $\tilde{G}=Spin(2n)$ is either $\mu_2\times\mu_2$ (if $2n=4k$) or $\mu_4$ (if $2n=4k+2$). The central subgroup corresponding to $SO(2n)$ is either the diagonal copy of $\mu_2\subset\mu_2\times\mu_2$ or the unique $\mu_2$ subgroup of $\mu_4$ -- in either case this subgroup is isotropic with respect to the natural pairing on $Z(\widetilde{G})$, and so induces an isotropic subgroup $H^1(C;\mu_2)\subset H^1(C;Z(\widetilde{G}))$. By nondegeneracy of the skew-pairing on $H^1(C;Z(\widetilde{G}))$ this subgroup is maximal isotropic, and the resulting abelian scheme $\frac{\Higgs_{Spin(2n)}(C)}{H^1(C;\mu_2)}$ is isomorphic to $\Higgs_{SO(2n)}^0(C)$, the moduli space of $SO(2n)$-Higgs bundles with vanishing second Stiefel-Whitney class.

To make this example extremely concrete, consider the first non-trivial case $G=SO(4)$. The universal cover is $\tilde{G}=Spin(4)=SU(2)\times SU(2)$ with centre $\mu_2\times\mu_2$, corresponding to the $\mu_2$ centres of each of the $SU(2)$ factors. $Spin(4)$ double covers the spaces $SO(3)\times SU(2)$, $SU(2)\times SO(3)$, and $SO(4)$, corresponding respectively to the subgroups $\mu_2 \times 1$, $1\times \mu_2$, and the diagonal subgroup $\Delta$. Denote the unique nondegenerate pairing on $\mu_2$ by $\Upsilon_2$; then the pairing on the central $\mu_2\times \mu_2$ is
\begin{align}\label{eq:induced-mod2-pairing}
	\Upsilon((a,b),(c,d))=\Upsilon_2(a,c)\Upsilon_2(b,d) .
\end{align}
On the diagonal subgroup corresponding to $SO(4)$, this pairing is identically 1, since $\Upsilon((a,a),(b,b))=\Upsilon_2(a,b)^2 =1$. Hence the subgroup $H^1(C;\Delta)\subset H^1(C;\mu_2\times \mu_2)$ is isotropic, and by nondegeneracy of the cup product pairing and of $\Upsilon$ on $\mu_2\times\mu_2$ it is maximal isotropic -- hence the results of the previous paragraph apply.
\end{eg}

\begin{eg}\label{eg:SO-Sp-dual}
Finally, it is interesting to consider what the duality of Theorem \ref{thm:dual-of-M} looks like for the simply-connected groups $Sp(2n)$ and $Spin(2n+1)$, whose Lie algebras are exchanged by Langlands duality.

First, consider the isomorphism
\begin{align}\label{eq:Spin-odd-dual}
	\left( \frac{\mc{M}_{\mf{sp}(2n)}(C)}{H^1(C;\mu_2)} \right)^D
		\cong\mc{M}_{\mf{so}(2n+1)}(C) .
\end{align}
The stack we are dualising on the left hand side of \eqref{eq:Spin-odd-dual} is the moduli space of $PSp(2n)=Sp(2n)/\mu_2$-Higgs bundles equipped with the gerbe of liftings of the universal $PSp(2n)$-Higgs bundle to a universal symplectic Higgs bundle. To interpret the right hand side, use the standard embedding $\mu_2=Z(Spin(2n+1)) \subset \mb{C}^\times$ to construct
\begin{align}\label{eq:Spin-c}
	\frac{Spin(2n+1)\times\mb{C}^\times}{\mu_2}= Spin^c(2n+1)_{\mb{C}}
\end{align}
the complexification of the compact group $Spin^c(2n+1)$. Fix a point $x\in C$. Then the moduli stack $\mc{M}_{\mf{so}(2n+1)}(C)$ may be identified as the stack of $Spin^c(2n+1)_{\mb{C}}$-Higgs bundles $(E,\phi)$ equipped with an isomorphism
\begin{align}\label{eq:push-Spinc}
\pd_\ast(E)\simeq\left\{\begin{array}{c c} \mc{O}_C & \text{or} \\ \mc{O}_C(x) \end{array}\right.
\end{align}
and with $\phi$ ``tracefree'' (cf. \eqref{diag:Higgs-stack-proj}). Specifically, the neutral component $\mc{M}_{\mf{so}(2n+1)}^0(C)$ may be identified with the usual moduli stack $\HiggsSt_{Spin(2n+1)}(C)$, and the non-neutral component $\mc{M}_{\mf{so}(2n+1)}^1(C)$ may be identified as the moduli stack of $Spin^c(2n+1)_{\mb{C}}$-Higgs bundles $(E,\phi)$ equipped with an isomorphism $\pd_\ast(E)\simeq \mc{O}_C(x)$. Since $H^2(C;\mu_2)=\mu_2$ we have that $\mc{M}_{\mf{so}(2n+1)}(C) = \mc{M}_{\mf{so}(2n+1)}^0(C)\coprod \mc{M}_{\mf{so}(2n+1)}^1(C)$.

Next consider the isomorphism
\begin{align}\label{eq:Sp-dual}
	\left( \frac{\mc{M}_{\mf{so}(2n+1)}(C)}{H^1(C;\mu_2)} \right)^D
		\cong\mc{M}_{\mf{sp}(2n)}(C) .
\end{align}
We have already seen one interpretation of the left hand side in terms of $Spin^c(2n+1)_{\mb{C}}$-Higgs bundles -- another interpretation is that on the left hand side we are dualising the moduli space of $SO(2n+1)$-Higgs bundles equipped with the gerbe of liftings of the universal $SO(2n+1)$-Higgs bundle to a universal $Spin(2n+1)$-Higgs bundle.

To interpret the right hand side we again construct the corresponding group $\widetilde{G}_\tau$ -- this time the group is
\begin{align}\label{eq:Sp-c} 
	GSp(2n)_{\mb{C}}:=\frac{Sp(2n,\mb{C})\times\mb{C}^\times}{\mu_2} ,
\end{align}
the \emph{general symplectic group} of linear automorphisms which preserve a given symplectic form up to a scaling factor. Then $\mc{M}_{\mf{sp}(2n)}(C)$ is -- imprecisely -- the stack of $GSp(2n)_{\mb{C}}$-Higgs bundles ``with fixed second Stiefel-Whitney class, considered up to parity''. The precise interpretation of the two connected components is analogous to the interpretation for $Spin(2n+1)$: $\mc{M}_{\mf{sp}(2n)}^0(C)$ is isomorphic to the moduli stack $\HiggsSt_{Sp(2n)}(C)$, and $\mc{M}_{\mf{sp}(2n)}^1(C)$ may be identified as the moduli stack of $GSp(2n)_{\mb{C}}$-Higgs bundles $(E,\phi)$ equipped with an isomorphism $\pd_\ast(E)\simeq \mc{O}_C(x)$, and satisfying $\tr(\phi)=0$.
\end{eg}

\label{d}

%% file: section-appendix-fourier-transform.tex

\section{Sheaves on $\mc{O}^\times$ gerbes and the Fourier-Mukai transform}\label{s:FM}

In this appendix, I wish to explain the (deceptively simple!) reason why one expects the Fourier-Mukai transform to send sheaves of weight +1 on $_\alpha X_\beta$ to sheaves of weight -1 on the dual $_\beta Y_\alpha$, where the notation agrees with that of Section \ref{s:stack-dual} and \cite{B-B09}. I do make one notational change in this section, suggestive of the idea that this appendix may be applicable in more situations than the one in this paper: I replace the sheaf $\mc{O}^\times$ with the multiplicative group $\mb{G}_m$.

Recall that a multiplicative $\mb{G}_m$ gerbe over a commutative group stack $A$ may be described as an extension of commutative group stacks
\begin{align}\label{eq:cgs-ext}
	0 \to B\mb{G}_m \to \mc{A}\to A \to 0.
\end{align}
Similarly, a given a torsor $A'$ for a commutative group stack $A$ we may construct an extension of $\mb{Z}$ by $A$
\begin{align}\label{eq:cgs-tors}
	0 \to A \to \widetilde{A} \to \mb{Z} \to 0
\end{align}
such that $A' \cong \widetilde{A}_1$ is the fibre over $1\in\mb{Z}$ \cite[Appendix]{DonPan08}. Since we are interested in $\mb{G}_m$ gerbes over torsors, and since the weight grading on the category of sheaves on a $\mb{G}_m$ gerbe is induced by the $B\mb{G}_m$ substack, in order to understand how the Fourier-Mukai transform acts on weights it is sufficient to examine the trivial case $\mb{Z}\times B\mb{G}_m$.

The categories of quasicoherent sheaves on $\mb{Z}$ and $B\mb{G}_m$ are both given by graded vector spaces, where the grading for $\mb{Z}$ is induced by the support of the sheaf and the grading for $B\mb{G}_m$ comes from the identification $\QCoh(B\mb{G}_m)=\Rep(\mb{G}_m)$.

Denote by $\mb{C}_m \in \QCoh(B\mb{G}_m)$ the irreducible $\mb{G}_m$-representation of weight $m$, and by $\mb{C}_{\{n\}}\in\QCoh(\mb{Z})$ the 1-dimensional sheaf supported on $n\in\mb{Z}$. Thinking of $\mb{Z}$ as $\Hom(B\mb{G}_m,B\mb{G}_m)$, the Poincar\'e sheaf $\mc{P}$ classified by the canonical evaluation map
\begin{align}\label{eq:}
	ev:\mb{Z}\times B\mb{G}_m \to B\mb{G}_m
\end{align}
is given by $\mc{P}=\prod_{n\in\mb{Z}} \mb{C}_{\{n\}}\boxtimes\mb{C}_n$. Consider the projections
\begin{equation}\label{diag:ZBG-proj}
\begin{tikzcd}
			& \mb{Z}\times B\mb{G}_m \ar[dl,"\pi"]\ar[dr,"\rho"'] \\
	\mb{Z}	&	& B\mb{G}_m
\end{tikzcd}
\end{equation}
The corresponding pushforwards are
\begin{align}
	\rho_\ast(V)	& = \prod_{n\in\mb{Z}} V_{\{n\}} \in \QCoh(B\mb{G}_m), 	\label{eq:rho-push} 	\\
	\pi_\ast(V) 	& = V^{\mb{G}_m} \in\QCoh(\mb{Z}). \label{eq:pi-push}
\end{align}
Now, let's consider the action of the integral transform on the simple objects $\mb{C}_{\{n\}}$ and $\mb{C}_m$ of our categories. We have
\begin{align}\label{eq:Z-FM-BG}
	\rho_\ast(\mc{P}\tens\pi^\ast \mb{C}_{\{n\}})
		= \rho_\ast(\mb{C}_n \boxtimes \mb{C}_{\{n\}})
		= \mb{C}_n
\end{align}
so that sheaves supported on $n\in\mb{Z}$ become weight $n$ $\mb{G}_m$-representations. On the other hand,
\begin{align}\label{eq:BG-FM-Z}
	\pi_\ast(\mc{P}\tens\rho^\ast \mb{C}_m)
		= \pi_\ast\left(\prod_{k\in\mb{Z}} \mb{C}_{\{k\}}\boxtimes \mb{C}_{k+m}\right)
		= \mb{C}_{\{-m\}},
\end{align}
so representations of weight $m$ become sheaves supported on $-m \in \mb{Z}$. Considering now the category of sheaves on $\mb{Z}\times B\mb{G}_m$, which is $\mb{Z}\times\mb{Z}$-graded by weight and support, we see that the Fourier-Mukai functor acts on $\mb{Z}\times\mb{Z}$ as
\begin{align}\label{eq:FM-support-weight}
	(m,n) \mapsto (n,-m),
\end{align}
and in particular $(-1,1)\mapsto (1,1)$. Since in \eqref{eq:cgs-tors} sheaves supported over $1\in\mb{Z}$ are exactly those sheaves supported on the torsor $A'$, we see that weight $-1$ sheaves on $A'$ are sent by the Fourier-Mukai transform to weight $+1$ sheaves on its dual.

\label{e}

%% file: section-appendix-Weyl-fixed.tex

\section{Fixed points of Weyl group actions}\label{s:Weyl}


Assume that $G$ is a simple and connected complex algebraic group, with a choice of maximal torus $H\subset G$. Via the exponential map we have an (analytic and $W$-equivariant) identification
\begin{align}\label{eq:H-from-lattice}
	H \cong \frac{X_\bullet(H)\tens\bC}{X_\bullet(H)} = X_\bullet(H)\tens\bC^\times
\end{align}
where $X_\bullet(H) = \Hom(\bC^\times, H) \subset \mf{h}$ is the cocharacter lattice of $H$.

Recall that the Weyl reflection $s_\alpha:\mf{h}^* \to \mf{h}^*$ corresponding to the root $\alpha$ is defined by\footnote{Recall that our convention is that $\alpha$ defines a character of $H$, hence its derivative $d\alpha$ defines a linear functional on $\mf{h}$.}
\begin{align}\label{eq:Weyl-reflection}
	s_\alpha(\lambda) = \lambda - \lambda(H_\alpha)d\alpha ,
\end{align}
where $H_\alpha$ is the coroot associated to $\alpha$, i.e.\ the unique element of $[\mf{g}_\alpha,\mf{g}_{-\alpha}]$ satisfying $d\alpha(H_\alpha)=2$. Dualising this, we have that $s_\alpha\in W_G(H)$ acts on $\mf{h}$ via
\begin{align}\label{eq:dual-Weyl-reflection}
	s_\alpha(x) = x-d\alpha(x)H_\alpha.
\end{align}
Translating this via the exponential map into a question about fixed points on the maximal torus $H$, we say that a point $x\in\mf{h}$ is a fixed point of $s_\alpha$ if and only if $s_\alpha(x) \in x + X_\bullet(H)$, which, using our explicit description of $s_\alpha$, occurs if and only if $d\alpha(x)H_\alpha \in X_\bullet(H)$.

\begin{prop}\label{prop:eval-of-fixed-pt}
If $h\in H$ is fixed by the action of $s_\alpha$, then $\alpha(h)=\pm 1$.
\end{prop}
\begin{proof}
Let $\Lambda_R$ denote the root lattice and $X^\bullet(G,H)=X^\bullet(H)$ the character lattice of $G$, both thought of as embedded in $\mf{h}^\ast$. We have
\begin{align}\label{eq:cochar-H}
	X_\bullet(H) = \{y\in\mf{h} \,|\, \lambda(y)\in\bZ \text{ for all } \lambda\in X^\bullet(H)\}.
\end{align}
Represent the fixed $h\in H$ by $x\in\mf{h}$. Since $\Lambda_R\subset X^\bullet(H)$ we have that $s_\alpha(x)\in x+ X_\bullet(H)$ implies $d\alpha(d\alpha(x)H_\alpha)\in\bZ$, equivalently $2d\alpha(x)\in\bZ$, and so $d\alpha(x)\in\frac{1}{2}\bZ$. But then for some $n\in\bZ$
\begin{align}\label{eq:root-on-fixed-pt}
	\alpha(t) = e^{2\pi i d\alpha(x)} = e^{\pi i n} \in \{\pm 1\}.
\end{align}
\end{proof}

Recall that if $G_1 \to G_2$ is an isogeny of simple groups inducing an isogeny on maximal tori $H_1\to H_2$, then $X_\bullet(G_1,H_1)\subset X_\bullet(G_2,H_2)$. This reflects the fact that if $x\in\mf{h}$ represents a fixed point of $s_\alpha$ acting on $H_1\subset G_1$, then it also represents a fixed point of $s_\alpha$ acting on $H_2 \subset G_2$. This is not a deep fact: the isogeny is $W$-equivariant, where $W\equiv W_{G_1}(H_1)=W_{G_2}(H_2)$, since it corresponds to the quotient by a central subgroup and the Weyl group action is induced by conjugation. More interesting is the question of when a fixed element $h_2\in H_2^{s_\alpha}$ can be \emph{lifted} to a fixed element $h_1\in H_1^{s_\alpha}$. It turns out that we can give a simple and exact answer to this question when the group we wish to lift to is the simply-connected form of the group.

\begin{prop}\label{thm:fixed-lifts}
Let $\tilde{G}$ be a simple, connected, simply-connected complex algebraic group, and let $\tilde{G}\to G$ be an isogeny of simple groups. Choose a maximal torus $\widetilde{H}\subset\widetilde{G}$ and denote by $H$ the corresponding maximal torus in $G$. Suppose that $h\in H$ is fixed by the root reflection $s_\alpha\in W$. Then a preimage $\tilde{h}\in \widetilde{H}$ of $h$ is fixed by $s_\alpha$ if and only if $\alpha(h)=1$.
\end{prop}

\begin{proof}
We first translate this into a statement about lattices and integrality: specifically the claim of the theorem is equivalent to the claim that for any element $x\in\mf{h}$ representing $h$, $d\alpha(x)\in\bZ$ if and only if $d\alpha(x)H_\alpha\in X_\bullet(\widetilde{G},\widetilde{H})$. In this form, the theorem follows from the fact that the cocharacter lattice for the simply connected form of the group is exactly the coroot lattice (i.e.\ the integral span of the coroots).
\end{proof}

\begin{remark}\label{rk:fixed-lifts-semisimple}
By considering products of simple groups and their Weyl groups, Proposition \ref{thm:fixed-lifts} immediately extends to all semi-simple complex algebraic groups.
\end{remark}

\begin{eg}\label{eg:Winvt-sl2}
Consider the groups $SL_2\bC$ and $PGL_2\bC$, with a simultaneous choice of Cartan subalgebra $\mf{h}=\{2\times 2 \text{ traceless complex matrices}\}$. Let $h=\left(\begin{array}{cc}1 & 0 \\0 & -1\end{array}\right)$, and consider the character
\begin{align}\label{eq:std-sl2-root}
	d\alpha:\mf{h}\to\bC	\\
	d\alpha\left(a\cdot h\right) = 2a \nonumber
\end{align}
Then the root, weight, and character lattices are given by
\begin{align}\label{eq:sl2-chars}
	\Lambda_R  = \bZ\cdot d\alpha = X^\bullet(PGL_2,H_{\ad})	\\
	\Lambda_W = \frac{1}{2}\bZ\cdot d\alpha = X^\bullet(SL_2,H) \nonumber
\end{align}
and the coroot, coweight, and cocharacter lattices are
\begin{align}\label{eq:sl2-cochars}
	\Pi_R = \bZ\cdot h = X_\bullet(SL_2,H)	\\
	\Pi_W = \frac{1}{2}\bZ\cdot h = X_\bullet(PGL_2,H) \nonumber
\end{align}
The Weyl group in this case is of order 2, with non-trivial element acting on $\mf{h}$ by $s_\alpha(x) = -x$, so that $x$ exponentiates to a fixed point in $G$ if and only if $2x\in X_\bullet(G,H)$. For $G=SL_2\bC$ this translates to $d\alpha(x)\in\bZ$, which upon exponentiating gives
	\[	\left(\begin{array}{cc} \pm 1 & 0 \\ 0 & \pm 1\end{array} \right).	\]
For $G=PGL_2\bC$ this translates to $d\alpha(x)\in\frac{1}{2}\bZ$, which upon exponentiating gives a new non-trivial fixed element given by the equivalence class of
	\[	\left(\begin{array}{cc} i & 0 \\ 0 &  -i \end{array} \right).	\]
\end{eg}

\label{f}

%% file: section-appendix-reductive-structure.tex

\section{Structure results for $\widetilde{G}_\tau$}\label{s:structure-results}

In this appendix I record some results on the structure of the reductive algebraic group $\widetilde{G}_\tau$, which was used in Section \ref{s:stack-dual} to construct the moduli stack $\mc{M}_{\widetilde{G}}(X)$.


\subsection{The Langlands dual of the map $\tau$}\label{ss:dualmap}

Consider the exact sequence of complex algebraic groups
\begin{align}\label{eq:G-tau-SES}
	1\to Z(\widetilde{G})\to \widetilde{G}\times T \to \widetilde{G}_\tau \to 1 .
\end{align}

\begin{prop}\label{p:dual-ses}
There is a dual exact sequence
\begin{align}\label{eq:G-tau-dual-SES}
	1\to Z(\widetilde{^LG})\to \phantom{.}^L(\widetilde{G}_\tau)\to (^LG)_{\ad}\times{^LT} \to 1 .
\end{align}
\end{prop}

\begin{proof}\label{proof:dual-ses}
Consider the exact sequence of \emph{abelian} groups
\begin{equation}\label{eq:tau-SES}
\begin{tikzcd}
	1 \ar[r] & Z(\widetilde{G})\ar[r,"\tau"] & T\ar[r] & T/Z(\widetilde{G})\ar[r] & 1 .
\end{tikzcd}
\end{equation}
Taking characters $\Hom(-,\mb{C}^\times)$ is a contravariant functor, and yields the exact sequence
\begin{align}\label{eq:character-SES}
	0\to X^\bullet(T/Z(\widetilde{G}))\to X^\bullet(T)\to Z(\widetilde{G})^\vee \to 0 ,
\end{align}
i.e.\
\begin{align}\label{eq:cocharacter-SES}
	0\to X_\bullet(\phantom{.}^L(T/Z(\widetilde{G})))\to X_\bullet(\phantom{.}^LT)\to Z(\widetilde{^LG}) \to 0 .
\end{align}
Apply $-\tens_{\bZ}^{\mathbf{L}}\mb{C}^\times$ and take homology to get the exact sequence
\begin{align}\label{eq:LT-tor-SES}
	1 \to\Tor_1^{\bZ}(Z(\widetilde{^LG}),\mb{C}^\times)\to\phantom{.}^L(T/Z(\widetilde{G}))\to\phantom{.}^LT\to 1 .
\end{align}
As an abelian group $\bC^\times \cong \mb{R}_{>0}^\times \times U(1) \cong \mb{R}\times U(1)$, and so $\Tor_1^{\bZ}(Z(\widetilde{^LG}),\mb{C}^\times)$ is canonically isomorphic to the torsion subgroup of $Z(\widetilde{^LG})$ (which is the entire group, since $Z(\widetilde{G})$ is torsion). Hence we have an exact sequence
\begin{equation}\label{eq:Ltau-SES}
\begin{tikzcd}
	1\ar[r] & Z(\widetilde{^LG})\ar[r,"^L\tau"] & ^L(T/Z(\widetilde{G}))\ar[r] & ^LT \ar[r] & 1 .
\end{tikzcd}
\end{equation}
Choose a maximal torus $\widetilde{H}\subset\widetilde{G}$. Via the above procedure the exact sequence
\begin{align}\label{eq:tau-Cartan-SES}
	1\to Z(\widetilde{G})\to \widetilde{H}\times T \to \frac{\widetilde{H}\times T}{Z(\widetilde{G})} \to 1
\end{align}
yields the exact sequence
\begin{align}\label{eq:Ltau-Cartan-SES}
	1\to Z(\widetilde{^LG})\to {^L\left(\frac{\widetilde{H}\times T}{Z(\widetilde{G})}\right)} \to \phantom{.}^L\widetilde{H}\times\phantom{.}^L T \to 1
\end{align}
and so via the inclusions $\widetilde{H}\subset\widetilde{G}$, $^L\left(\frac{\widetilde{H}\times T}{Z(\widetilde{G})}\right)\subset{^L(\widetilde{G}_\tau)}$, the exact sequence \eqref{eq:G-tau-SES} induces the dual exact sequence \eqref{eq:G-tau-dual-SES}.
\end{proof}


\subsection{Structure of the Langlands dual group}\label{ss:dualstructure}
There is another inclusion
\begin{equation}\label{eq:1xtau}
\begin{tikzcd}[row sep = tiny]
	Z(\widetilde{G})\ar[r,"1\times\tau"]	& \widetilde{G}\times T 	\\
	z \ar[r,mapsto]							& (1_{\widetilde G},\tau(z))
\end{tikzcd}
\end{equation}
which induces an exact sequence
\begin{equation}\label{eq:G-tau-proj-SES}
\begin{tikzcd}
	1\ar[r]
		& Z(\widetilde{G})\ar[r,"\overline{1\times\tau}"]
			& \widetilde{G}_\tau \ar[r] & G_{\ad}\times (T/Z(\widetilde{G}))\ar[r]
				& 1 .
\end{tikzcd}
\end{equation}
\begin{prop}\label{prop:Langlands-dual-SES}
The Langlands dual exact sequence is given by
\begin{equation}\label{eq:LG-tau-SES}
\begin{tikzcd}
	1\ar[r]
		& Z(\widetilde{^LG})\ar[r,"{^L\iota}\times{^L\tau}"]
			& \widetilde{^LG}\times{^L(T/Z(\widetilde{G}))}\ar[r]
				& ^L(\widetilde{G}_\tau) \ar[r] & 1
\end{tikzcd}
\end{equation}
where $\iota:Z(\widetilde{G})\subset\widetilde{G}$ and $^L\iota:Z(\widetilde{^LG})\subset\widetilde{^LG}$ are the subgroup inclusions, and $^L\tau$ is the map embedding of Proposition \ref{p:dual-ses}. I.e.\ the Langlands dual of $\widetilde{G}_\tau$ is
\begin{align}\label{eq:LG-tau-is-G-Ltau}
	^L(\widetilde{G}_\tau)\cong\frac{\widetilde{^LG}\times{^L(T/Z(\widetilde{G}))}}{Z(\widetilde{^LG})} = (\widetilde{^LG})_{^L\tau} .
\end{align}
\end{prop}
\begin{proof}
It suffices to prove the result after replacing the group $\widetilde{G}$ with a choice of maximal torus $\widetilde{H}$. Consider the following commutative diagram, where all rows and columns are exact:
\begin{equation}\label{diag:tau-torus-square}
\begin{tikzcd}
		& 0\ar[d]
			& 1\ar[d]
				& 1\ar[d]
					&	\\
	1\ar[r]
		& Z(\widetilde{G})\ar[r,"\tau"]\ar[d,equals]
			& T\ar[r]\ar[d,"\overline{1\times\id}"]
				& T/Z(\widetilde{G})\ar[r]\ar[d,"1\times\id"]
					& 1 	\\
	1\ar[r]
		& Z(\widetilde{G})\ar[r,"\overline{1\times\tau}"]\ar[d]
			& \frac{\widetilde{H}\times T}{Z(\widetilde{G})}\ar[r]\ar[d]
				& H_{\ad}\times T/Z(\widetilde{G})\ar[r]\ar[d]
					& 1 	\\
		& 0\ar[r]
			& H_{\ad}\ar[r,equals]\ar[d]
				& H_{\ad}\ar[r]\ar[d]
					& 0 	\\
		&	& 1 & 1 &
\end{tikzcd}
\end{equation}
Applying $(-)^\vee:=\Hom(-,\mb{C}^\times)$ yields another commutative diagram, again with all rows and columns exact:
\begin{equation}\label{diag:tau-character-square}
\begin{tikzcd}
		& 0\ar[d]
			& 0\ar[d]
				&	&	\\
	0\ar[r]
		& X^\bullet(H_{\ad})\ar[r,equals]\ar[d]
			& X^\bullet(H_{\ad})\ar[r]\ar[d]
				& 0\ar[d] &	\\
	0\ar[r]
		& X^\bullet(H_{\ad})\times X^\bullet(T/Z(\widetilde{G}))\ar[r]\ar[d]
			& X^\bullet\left(\frac{\widetilde{H}\times T}{Z(\widetilde{G})}\right)\ar[d]\ar[r,"\overline{1\times\tau}^\vee"]
				& Z(\widetilde{G})^\vee \ar[r]\ar[d,equals]
					& 0 \\
	0\ar[r]
		& X^\bullet(T/Z(\widetilde{G}))\ar[r]\ar[d]
			& X^\bullet(T)\ar[r,"\tau^\vee"]\ar[d]
				& Z(\widetilde{G})^\vee \ar[r]\ar[d]
					& 0 \\
		& 0 & 0 & 0 &
\end{tikzcd}
\end{equation}
Applying $-\tens_{\bZ}^{\mathbf{L}}\mb{C}^\times$ and taking homology yields a third commutative diagram with all rows and columns exact:
\begin{equation}\label{diag:tau-tor-square}
\begin{tikzcd}
		&	& 1\ar[d]
				& 1\ar[d]
					&		\\
		& 0\ar[d]\ar[r]
			& \widetilde{^LH}\ar[r,equals]\ar[d]
				& \widetilde{^LH}\ar[r]\ar[d]
					& 0 	\\
	1\ar[r]
		& Z(\widetilde{^LG})\ar[d,equals]\ar[r,"^L\overline{1\times\tau}"]
			& \widetilde{^LH}\times{^L(T/Z(\widetilde{G}))}\ar[r]\ar[d]
				& ^L\left(\frac{\widetilde{H}\times T}{Z(\widetilde{G})}\right)\ar[r]\ar[d]
					& 1 	\\
	1\ar[r]
		& Z(\widetilde{^LG})\ar[r,"^L\tau"]\ar[d]
			& ^L(T/Z(\widetilde{G}))\ar[r]\ar[d]
				& ^LT\ar[r]\ar[d]
					& 1 	\\
		& 0 & 1 & 1 &
\end{tikzcd}
\end{equation}
Therefore, composing $^L\overline{1\times \tau}$ with projection to the second factor gives
\begin{equation}\label{diag:proj-Ltau}
\begin{tikzcd}
	Z(\widetilde{^LG})\ar[r,"^L\overline{1\times\tau}"]\ar[rd,"^L\tau"']
		& \widetilde{^LH}\times{^L(T/Z(\widetilde{G}))}\ar[d]	\\
		& ^L(T/Z(\widetilde{G}))
\end{tikzcd}
\end{equation}
Repeating this argument but with the central column in the first diagram given by
\begin{equation}\label{eq:Cartan-incl-SES}
\begin{tikzcd}
	1\ar[r] & \widetilde{H}\ar[r]
			& \frac{\widetilde{H}\times T}{Z(\widetilde{G})}\ar[r]
			& T/Z(\widetilde{G})\ar[r] & 1
\end{tikzcd}
\end{equation}
shows that composition with the first projection is
\begin{equation}\label{diag:proj-Lincl}
\begin{tikzcd}
	Z(\widetilde{^LG})\ar[r,"^L\overline{1\times\tau}"]\ar[rd,"^L\iota"']
		& \widetilde{^LH}\times{^L(T/Z(\widetilde{G}))}\ar[d]	\\
		& \widetilde{^LH}
\end{tikzcd}
\end{equation}
Therefore, $^L\overline{1\times\tau} = {^L\iota}\times{^L\tau}$.
\end{proof}

\label{g}